\documentclass[11pt,tbtags]{amsart}%
\usepackage{amsfonts}
\usepackage{amsmath}
\usepackage{amssymb}
\usepackage{graphicx}
\providecommand{\U}[1]{\protect\rule{.1in}{.1in}}
\providecommand{\U}[1]{\protect\rule{.1in}{.1in}}
\newtheorem{theorem}{Theorem}[section]

\newtheorem{lemma}[theorem]{Lemma}

\theoremstyle{definition}
\newtheorem{definition}[theorem]{Definition}
\newtheorem{remark}{Remark}

\DeclareMathOperator{\dv}{div}
\DeclareMathOperator{\meas}{meas}

\DeclareMathOperator{\sign}{sign}

\allowdisplaybreaks[2]
\numberwithin{equation}{section}

\title[Uniqueness of renormalized solutions to nonlinear parabolic
\ldots]{Uniqueness of renormalized solutions to nonlinear parabolic
  problems with lower order terms }
\author{Rosaria Di Nardo}
\address{Dipartimento di Matematica  \\
  Seconda Universit\`a degli Studi di Napoli\\
  viale Lincoln, 5 - 81100 CASERTA, Italy}
\email{Rosaria.Dinardo@unina.it}

\author{Filomena Feo}
\address{Dipartimento per le Tecnologie, Universit\`a degli Studi di
  Napoli “Parthenope”, Centro Direzionale Isola C4, 80100 Naples,
  Italy}
\email{Filomena.Feo@uniparthenope.it}
\author{Olivier Guib\'e}
\address{Laboratoire de Math\'ematiques Rapha\"el Salem \\
Universit\'e de Rouen, CNRS \\
F-76801 Saint Etienne du Rouvray cedex, France}
\email{Olivier.Guibe@univ-rouen.fr}
\begin{document}
\maketitle

\begin{abstract}
In this paper we prove uniqueness results for renormalized solutions to a
class of nonlinear parabolic problems.

\end{abstract}

\section{Introduction}
In the present paper we investigate  the uniqueness of the following class of nonlinear parabolic problems
\begin{equation}
\left\{
\begin{aligned} &\frac{\partial u}{\partial
    t}-\dv\left( a(x,t,u,\nabla u)\right) \\ & \quad \qquad {}+\dv
  \left( K(x,t,u)\right) +H(x,t,\nabla u)=f- \dv g & \text{ in }
  Q_{T}\\ & u(x,t)=0 & \text{\hglue-1cm
    on } \partial\Omega\times\left( 0,T\right) \\ &u(x,0)=u_{0}(x) &
  \text{ in } \Omega, \end{aligned}\right.
\label{problema}%
\end{equation}
where $Q_{T}$ is the cylinder $\Omega\times(0,T)$, $\Omega$ is a bounded open
subset of $\mathbb{R}^{N}$, $T>0$,
$p>1$ and $N\geq2$. Moreover $ -\dv\left( a(x,t,u,\nabla u)\right)$ is
a Leray-Lions operator which is coercive and grows like $|\nabla
u|^{p-1}$ with respect to $\nabla u$. The function $K$ and $H$ are
Carath\'{e}odory functions with suitable assumptions (see Theorems
\ref{th1}, \ref{th2} and \ref{th3}). Finally $f\in
L^{1}(Q_{T})$, $g\in( L^{p^{\prime}}(Q_{T})) ^{N}$ and $u_{0}\in
L^{1}(\Omega)$.

The difficulties connected to existence and uniqueness of the solution
to this problem are due to the $L^{1}$ data and to
the presence of the two terms $K$ and $H$ which can induce a lack of
coercivity.

For $L^{1}$ data  and $p>2-\frac{1}{N+1}$ the existence of a weak
solution to Problem~(\ref{problema}) (which belongs to
$L^{m}((0,T);W_{0}^{1,m}(\Omega))$ with $m<\frac{p(N+1)-N}{N+1}$) was
proved in \cite {boccardogallouet1} (see also
\cite{bocc-all'aglio-galue-orsina}) when $K\equiv H\equiv 0$
and in \cite{porzio} when $K\equiv 0$. It is well known that this weak
solution is not unique in general (see \cite{serrin} for a
counter-example in the stationary case). In the present paper we use the
framework of renormalized solutions which provides uniqueness and
stability properties.

The notion of renormalized solution was introduced in \cite{DPL89a,DPL89b}
for first order equations and has been adapted for elliptic problems with $%
L^{1}$ data in \cite{muratpreprint,murat94}) and with
bounded measure data in \cite{dmop}. This notion was also developed for
parabolic equation with $L^{1}$ data in
\cite{blanchard-murat,blanchard-murat-redvan} (see also \cite{petitta}
for measure data). Recall that the equivalent notion of entropy
solution for $L^{1}$ data was also developed for elliptic
equation in \cite{6nomi} (see also \cite{prignet} in the parabolic case).

In the case where $H\equiv 0$ and where the function $K(x,t,u)$ is
independent  on the $(x,t)$ variable and continuous,
the existence of a renormalized solution to Problem \eqref{problema} is proved in
\cite{blanchard-murat-redvan}. The case $H\equiv 0$, $g\equiv 0$ (and
where $K$  depends on $(x,t)$ and $u$) is investigated in
\cite{[dinardo]}. In \cite{d-f-g} the authors prove the existence of a
renormalized solution for the complete operator.

As far as the uniqueness of  renormalized solution to parabolic
equation  is concerned, we refer mainly to
\cite{blanchard-murat,blanchard-murat-redvan,carrillo-wittbold} where
in short the function $K$ does not depend on the $(x,t)$ variable and where
$H\equiv 0$ (see also \cite{blanchard-porretta05} for Stefan problem
with $L^{1}$ data). In particular, when $H\equiv 0$ and under a local
Lipschitz assumption on
$a(x,t,r,\xi)$ and on $K(r)$ with respect to $r$ the authors prove in
\cite{blanchard-murat-redvan} that the renormalized solution to
Problem \eqref{problema} is unique. With respect to the mentioned
references, the main novelty of the present
paper is to present uniqueness results to parabolic equations
\eqref{problema} with the two terms $-\dv(K(x,t,u))$ and $H(x,t,\nabla u)$.
The first result (see Theorem~\ref{th1}) deals with the case $H\equiv
0$ and establishes the
uniqueness of the renormalized solution to Problem \eqref{problema} under
a local Lipschitz condition on $a(x,t,r,\xi)$ and $K(x,t,r)$ with respect
to $r$. The proof uses the techniques developed in
\cite{blanchard-murat-redvan} and  the dependence of the function $K$
with respect to the $(x,t)$ variable leads to additional
difficulties here. Such difficulties are
overcome by a technical lemma (see Lemma \ref{Lemma1}) which specifies
the asymptotic behavior of some terms which appear in the uniqueness
process. The
second result (see Theorem \ref{th2} for $p\geq 2$ and Theorem
\ref{th3} for $2-\frac{1}{N+1}<p<2$) addresses Equation \eqref{problema}
with the presence of the two terms  $-\dv(K(x,t,u))$ and $H(x,t,\nabla
u)$. Under more restrictive assumptions on $a$ and under global
Lipschitz type condition on $K(x,t,s)$ with respect to $s$ and
$H(x,t,\xi)$ with respect to $\xi$ we show the uniqueness of the
renormalized solution. The proof uses two technical lemmas (Lemma
\ref{Lemma1} and Lemma \ref{lemma tecnico}) and the techniques
developed in \cite{d-f-g} for the existence of a solution to Problem
\eqref{problema} (see also \cite{porzio}).
We underline that we don't make any assumptions on the
smallness of the coefficients. Indeed for the analogous elliptic equation
with two lower order terms (see e.g. \cite{GM3} and
\cite{dinardo-perrotta}) it is necessary to
assume that one of the terms $K$ or $H$ is small enough in order to
obtain an existence and uniqueness result.

The paper is organized as follows. In Section~2  we present the
assumptions on the data  and we recall the definition of a
renormalized solution to
Problem \eqref{problema}. In Section~3 we state the main results of the
present paper.
Section~4 is devoted to the proof of the uniqueness
results.

\section{Assumptions and Definitions}

In this section we recall the definition of a renormalized solution to
nonlinear parabolic problems with lower order terms and $L^{1}(\Omega
\times(0,T))+L^{p^{\prime}}((0,T);W^{-1,p^{\prime}}(\Omega))$ data.

More precisely we consider the following problem
\begin{equation}
\left\{
\begin{aligned} &\frac{\partial u}{\partial t}-\dv\left( a(x,t,u,\nabla u)\right) \\ & \quad \qquad {}+\dv \left( K(x,t,u)\right) +H(x,t,\nabla u)=f- \dv g & \text{ in } Q_{T}\\ & u(x,t)=0 & \text{\hglue-1cm on } \partial\Omega\times\left( 0,T\right) \\ &u(x,0)=u_{0}(x) & \text{ in } \Omega, \end{aligned}\right.
\label{problema compl}%
\end{equation}
where $Q_{T}$ is the cylinder $\Omega\times(0,T)$, $\Omega$ is a bounded open
subset of $\mathbb{R}^{N}$ with boundary $\partial\Omega$, $T>0$,
$p>1$ and $N\geq2$.

The following assumptions hold true:

\begin{itemize}
\item $a:Q_{T}\times\mathbb{R}\times\mathbb{R}^{N}\rightarrow\mathbb{R}^{N}$
is Carath\'{e}odory function such that
\begin{equation}
a(x,t,s,\xi)\xi\geq\alpha_{0}\left\vert \xi\right\vert ^{p},\text{
\ \ \ \ }\alpha_{0}>0, \label{ellit}%
\end{equation}
and
\begin{equation}
\left(  a(x,t,s,\xi)-a(x,t,s,\overline{\xi}),\xi-\overline{\xi}\right)
>0, \label{monotonia}%
\end{equation}
for a.e. $(x,t)\in Q_{T}$, for any $s\in \mathbb{R}$ and any
$\xi,\overline{\xi} \in
\mathbb{R}^{N}$ with $\xi\neq \overline{\xi}$.\par
Moreover for any $k>0$ there exists $\beta_{k}>0$ and $h_{k}\in L^{p^{\prime}}(Q_{T})$
such that%
\begin{equation}
\left\vert a(x,t,s,\xi)\right\vert \leq h_{k}+\beta_{k}\left\vert
\xi\right\vert ^{p-1},\text{ \ \ \ for every }s\text{ such that }\left\vert
s\right\vert \leq k, \label{crescita a}%
\end{equation}%
for a.e. $(x,t)\in Q_{T}$ and any $\xi\in
\mathbb{R}^{N}$;
\item $K:Q_{T}\times\mathbb{R}\rightarrow\mathbb{R}^{N}$ is a Carath\'{e}odory
function such that
\begin{equation}
\left\vert K(x,t,s)\right\vert \leq c(x,t)(\left\vert s\right\vert
^{\gamma}+1), \label{crescita K}%
\end{equation}
with%
\begin{equation}
\gamma=\frac{N+2}{N+p}(p-1)\text{ and }c\in L^{r}(Q_{T})\text{ with }%
r\geq\frac{N+p}{p-1}, \label{sommabilit_c}%
\end{equation}
for a.e. $(x,t)\in Q_{T}$, for every $s\in\mathbb{R}$;

\item $H:Q_{T}\times\mathbb{R}^{N}\rightarrow\mathbb{R}$ is a Carath\'{e}odory
function such that
\begin{equation}
\left\vert H(x,t,\xi)\right\vert \leq b(x,t)(\left\vert \xi\right\vert
^{\delta}+1), \label{crescita H}%
\end{equation}
with%
\begin{equation}
\delta=\frac{N(p-1)+p}{N+2}\text{ and }b\in L^{N+2,1}(Q_{T}),
\label{sommabilit_b}%
\end{equation}
for a.e. $(x,t)\in Q_{T}$ and for every $\xi
\in\mathbb{R}^{N}$.
\par
Moreover we assume that
\begin{equation}
f\in L^{1}(Q_{T}), \label{f}%
\end{equation}%
\begin{equation}
g\in\big(L^{p^{\prime}}(Q_{T})\big)^{N} \label{dato g}%
\end{equation}
and%
\begin{equation}
u_{0}\in L^{1}(\Omega). \label{uzero}%
\end{equation}

\end{itemize}

Under these assumptions, the above problem does not admit, in general, a
solution in the sense of distribution since we cannot expect to have the
fields $a(x,t,u,\nabla u)$, $K(x,t,u)$ in $(L_{loc}^{1}(Q_{T}))^{N}$ and
$H(x,t,\nabla u)$ in $L_{loc}^{1}(Q_{T})$. For this reason in the present
paper we consider the framework of renormalized solutions.

For any $k>0$ we denote by $T_{k}$ the truncation function at height $\pm k$,
$T_{k}(s)=\max(-k,\min(k,s))$ for any $s\in\mathbb{R}$.

\medskip We recall the definition of a renormalized solution (see
\cite{blanchard-murat, blanchard-murat-redvan}) to Problem
(\ref{problema compl}).

\begin{definition}
A real function $u$ defined in $Q_{T}$ is a renormalized solution of $\left(
\ref{problema compl}\right)  $ if it satisfies the following conditions:%
\begin{equation}
u\in L^{\infty}((0,T);L^{1}(\Omega)), \label{87}%
\end{equation}%
\begin{equation}
T_{k}(u)\in L^{p}((0,T);W_{0}^{1,p}(\Omega)),\text{ for any }k>0, \label{9}%
\end{equation}%
\begin{equation}
\lim_{n\rightarrow+\infty}\frac{1}{n}\int_{\{(x,t)\in Q_{T}\,:\,|u(x,t)|\leq
n\}}a(x,t,u,\nabla u)\nabla udxdt=0, \label{12}%
\end{equation}
and if for every function $S\in W^{2,\infty}(\mathbb{R})$ which is piecewise
$C^{1}$ and such that $S^{\prime}$ has a compact support
\begin{multline}
\frac{\partial S(u)}{\partial t}-\dv(a(x,t,u,\nabla u)S^{\prime}%
(u))+S^{\prime\prime}(u)a(x,t,u,\nabla u)\nabla u\label{58}\\
{}+\dv(K(x,t,u)S^{\prime}(u))-S^{\prime\prime
}(u)K(x,t,u)\nabla u\\
{} + H(x,t,\nabla u)S^{\prime}(u) =fS^{\prime}(u)-(\dv g)S^{\prime}(u)\quad\text{in }\mathcal{D}^{\prime
}(Q_{T})
\end{multline}
and
\begin{equation}
S(u)(t=0)=S(u_{0})\quad\text{in }\Omega. \label{80}%
\end{equation}

\end{definition}

\begin{remark}
It is well known that conditions (\ref{87}) and (\ref{9}) allow to define
$\nabla u$ almost everywhere in $Q_{T}$: for any $k>0$ we have $\nabla
T_{k}(u)=\chi_{\{|u|<k\}}\nabla u$ a.e in $Q_{T}$ where $\chi_{\{|u|<k\}}$
denotes the characteristic function of the set $\{(x,t) :|u(x,t)|<k\}$. We
notice that equation (\ref{58}) can be formally obtained through pointwise
multiplication of (\ref{problema compl}) by $S^{\prime}(u)$ and all terms
except $S(u)_{t}$ in (\ref{58}) belong to $L^{1}(Q_{T})+L^{p^{\prime}%
}((0,T);W^{-1,p^{\prime}}(\Omega))$ since $T_{k}(u)\in L^{p}((0,T);W_{0}%
^{1,p}(\Omega))$, for any $k>0$ and $S^{\prime}$ has a compact support. It
follows that (\ref{58}) has a meaning in $\mathcal{D}^{\prime}(Q_{T})$ and
that the initial condition (\ref{80}) makes sense. At last condition
(\ref{12}) gives additional information on $\nabla u$ for large value of $|u|$.
\end{remark}

\medskip We use in the present paper the two Lorentz spaces $L^{q,1}(Q_{T})$
and $L^{q,\infty}(Q_{T})$, see for example \cite{lorentz,oneil} for
references about Lorentz spaces $L^{q,s}$. If $f^{\ast}$ denotes the
decreasing rearrangement of a measurable function $f$,
\[
f^{\ast}(r)=\inf\{s\geq0~:~\hbox {meas}\left\{  (x,t)\in Q_{T}%
~:~|f(x,t)|>s\right\}  <r\},
\]
with $r\in\lbrack0,\meas(Q_{T})]$,
$L^{q,1}(Q_{T})$ is the space of Lebesgue measurable functions such that
\[
\Vert f\Vert_{L^{q,1}(Q_{T})}=\left(  \int_{0}^{\meas(Q_{T})}f^{\ast
}(r)r^{\frac{1}{q}}\frac{dr}{r}\right)  <+\infty
\]
while $L^{q,\infty}(Q_{T})$ is the space of Lebesgue measurable functions such
that
\[
\Vert f\Vert_{L^{q,\infty}(Q_{T})}=\sup_{r>0}r\left[  \hbox {meas}\left\{
(x,t)\in Q_{T}~:~|f(x,t)|>r\right\}  \right]  ^{1/q}<+\infty.
\]
If $1<q<+\infty$ we have the generalized H\"{o}lder in$_{{}}$equality
\begin{equation}
\left\{
\begin{aligned} {} & \forall f\in L^{q,\infty}(Q_{T}),~\forall g\in L^{q',1}(Q_{T}), \\ {} & \displaystyle \int_{Q_{T}} |fg |\le \| f\|_{L^{q,\infty}(Q_{T})}\| g\|_{L^{q',1}(Q_{T})} .\label{lorh} \end{aligned}\right.
\end{equation}

Under the assumptions \eqref{ellit}-\eqref{uzero} the existence of a
renormalized solution to Problem \eqref{problema compl} is established
in \cite{d-f-g} and it is well known that \eqref{87}-\eqref{12} lead to
\begin{equation}
\left\vert \nabla u\right\vert \in L^{\frac{N(p-1)+p}{N+1},\infty}%
(Q_{T})\label{stimagradiente}%
\end{equation}%
and
\begin{equation}
u\in L^{\frac{N(p-1)+p}{N},\infty}(Q_{T}).\label{stimau}%
\end{equation}
Moreover  the growth assumptions \eqref{crescita K},
\eqref{crescita H} on $K$ and $H$, the regularities
\eqref{sommabilit_c}, \eqref{sommabilit_b} of $c$ and $b$
together with \eqref{87} and \eqref{12} allow to prove (see
\cite{d-f-g}) that any
renormalized solution to Problem \eqref{problema compl} verifies
\begin{equation}
  \label{eqoga}
  H(x,t,\nabla u)\in L^{1}(Q_{T})
\end{equation}
and
\begin{equation}
  \label{eqogb}
  \lim_{n\rightarrow +\infty} \frac{1}{n}\int_{\{(x,t)\in Q_{T}\,;\,|u(x,t)|<n\}}^{ }
  |K(x,t,u)| \, |\nabla u | dxdt=0.
\end{equation}
Properties \eqref{eqoga} and \eqref{eqogb} are crucial to obtain
uniqueness results.

\par\medskip

{\slshape Notation.} Throughout the paper, for the sake of shortness
if  $u$ is a measurable function defined on $Q_{T}$, we denote by
$\{ |u|\leq k\}$ (resp. $\{ |u| <k \}$) the measurable subset
$\{ (x,t)\in Q_{T}\,;\, |u(x,t)|\leq k\}$ (resp. $\{ (x,t)\in
Q_{T}\,;\, |u(x,t)|<k \}$. Moreover the explicit dependence in $x$ and
$t$ of the
functions $a$, $K$ and $H$  will be omitted
so that $a(x,t,u,\nabla
u)=a(u,\nabla u)$, $K(u)=K(x,t,u)$ and $H(\nabla u)=H(x,t,\nabla u)$.

\section{Statement of the results}

\subsection{First case: $H\equiv 0$}

In order to prove uniqueness result in the case $H(x,t,\xi)=0$ we
assume the further condition that
$a(x,t,s,\xi)$ and $K(x,t,s)$ are locally continuous Lipschitz with respect to $s$ :
for any compact set $C$ of $\mathbb{R}$, there exists
$L_{C}$ belonging to $L^{p'}(Q_{T)}) $ and $\gamma_{C}>0$
such that $\forall s,\overline{s}\in C$
\begin{gather}
  \label{loclipa}
  |a(x,t,s,\xi)-a(x,t,\overline{s},\xi)|\le
  \big( L_{C}(x,t) + \gamma_{C} |\xi|^{p-1}\big)
  |s-\overline{s}|
  \\
\left\vert K(x,t,s)-K(x,t,\overline{s})\right\vert    \leq
L_{C}(x,t)\left\vert s - \overline{s}\right\vert \label{loc lip K}
\end{gather}
for almost every $(x,t)\in Q_{T}$  and for every $\xi\in\mathbb{R}^{N}$.
\par
\smallskip
The main result of this subsection is the following theorem.

\begin{theorem}
\label{th1} Under the assumptions \eqref{ellit}-\eqref{sommabilit_c},
\eqref{f}-\eqref{uzero}, \eqref{loclipa} and \eqref{loc lip K}, the
renormalized solution to Problem \eqref{problema compl} is unique.
\end{theorem}

\subsection{Second case: general operator}

In order to prove uniqueness result for Problem (\ref{problema compl})
with the term $H(x,t,\nabla u)$ we assume in this subsection that the
function $a$ is independent of $r$ and is strongly monotone (see
assumptions \eqref{eqog11} in Theorem \ref{th2} and \eqref{eqog11a} in
Theorem \ref{th3}).

Moreover the functions $K(x,t,s)$ (resp. $H(x,t,\xi)$) is locally
Lipschitz continuous with respect to $s$ (resp. $\xi$) with a global
control of the Lipschitz coefficient:
\begin{equation}
\left\vert K\left(  x,t,s\right)  -K\left(  x,t,\overline{s}\right)  \right\vert \leq
c(x,t) \left(  1+\left\vert s\right\vert +\left\vert \overline{s}\right\vert
\right)  ^{\tau} |  s-\overline{s} | \text{ \ \ \ }\tau\geq0
\label{lip K}%
\end{equation}
and%
\begin{equation}
\left\vert H\left(  x,t,\xi\right)  -H\left(  x,t,\overline{\xi}\right)
\right\vert \leq b(x,t)  \left(  1+\left\vert \xi\right\vert +\left\vert
\overline{\xi}\right\vert \right)  ^{\sigma} |  \xi-\overline{\xi}|
  \ \ \ \sigma\geq0 \label{Lip H}%
\end{equation}
for a.e. $(x,t)\in Q_{T}$, for every $s,\overline{s}\in\mathbb{R}$, for every
$\xi,\overline{\xi}\in\mathbb{R}^{N}$ with $c\in L^{r,1}(Q_{T})$ and $b\in
L^{\lambda,1}(Q_{T})$ where $r$ and $\lambda$ belong to suitable intervals
(see Theorems 3.3 and 3.4)

We investigate the case $p\geq2$ and the case $2-\frac{1}{N+1}<p<2$ in two
different results.

\begin{theorem}
\label{th2} Let $p\geq2$. Let us assume that
\eqref{ellit}-\eqref{uzero} hold and that the function $a$ is
independent of $r$ and satisfies
\begin{equation}
\label{eqog11}
\left(a( x, t,\xi) -a(x,t,\overline{\xi})
 \right)  {    \cdot  }( \xi - \overline{\xi})    \geq
   \beta( 1 + \left\vert \xi\right\vert {   + }\left\vert
\overline{\xi}\right\vert {   )}^{p-2}{   }\left\vert \xi
-\overline{\xi}\right\vert ^{2}
\end{equation}
for a.e. $(x,t)\in Q_{T}$,  for every
$\xi,\overline{\xi}\in\mathbb{R}^{N}$ with $\xi\neq\overline{\xi}$ and
$\beta>0$.
\par
Moreover we assume that \eqref{lip K} and \eqref{Lip H} are
satisfied with
\begin{equation}
\begin{cases}
r\geq N+2, &
0\leq\tau\leq\dfrac{N(p-1)+p}{N}\Big(\dfrac{1}{N+2}-\dfrac{1}{r}\Big),
\\
\lambda\geq N+2, &
 0\leq\sigma\leq\dfrac{N(p-1)+p}{N+1}\Big(\dfrac{1}{N+2}-\dfrac{1}{\lambda}\Big).
 \end{cases}
 \label{ip 1}
\end{equation}
Then the renormalized solution to Problem $(\ref{problema compl})$ is unique.
\end{theorem}

\begin{theorem}
\label{th3} Let $2-\frac{1}{N+1}<p<2.$ Let us assume that
\eqref{ellit}-\eqref{crescita a}, \eqref{f}-\eqref{uzero} hold
and that the function $a$ is independent of $r$ and satisfies
\begin{equation}
\left(a( x, t,\xi) -a(x,t,\overline{\xi})
 \right)  {    \cdot  }( \xi - \overline{\xi})    \geq
   \beta\frac{\left\vert \xi-\overline{\xi}\right\vert ^{2}%
}{\left(  \left\vert \xi\right\vert +\left\vert \overline{\xi}\right\vert
\right)  ^{2-p}}
\label{eqog11a}%
\end{equation}
for a.e. $(x,t)\in Q_{T}$, for every
$\xi,\overline{\xi}\in\mathbb{R}^{N}$ with $\xi\neq\overline{\xi}$ and
$\beta>0.$
Moreover we assume that \eqref{lip K} and \eqref{Lip H} are satisfied with%
\begin{equation}
\begin{cases}
 \displaystyle
r>\dfrac{p(N+1)-N}{(p-1)(N+1)-N},\\[.3cm]
0\leq\tau<\dfrac{N(p-1)+p}{N}\left(  \dfrac{(p-1)(N+1)-N}{p(N+1)-N}-\dfrac{1}%
{r}\right),  \\
\lambda>\frac{p(N+1)-N}{(p-1)(N+1)-N},\\[.3cm]
\displaystyle 0\leq\sigma<\frac{N(p-1)+p}{N+1}\left(  \frac{(p-1)(N+1)-N}{p(N+1)-N}-\frac
{1}{\lambda}\right). \\[.3cm]
\end{cases}
  \label{ip 2}%
\end{equation}
Then  the renormalized solution to Problem $(\ref{problema compl})$ is
unique.
\end{theorem}

\begin{remark}
Let us compare the assumptions $(\ref{crescita K})$ and  $(\ref{crescita H})$ on the growth condition and the assumptions $(\ref{lip K})$ and $(\ref{Lip H})$ on the locally Lipschitz continuity made on $K(x,t,s)$ and $H(x,t,\xi)$ respectively. Observe that assumption $(\ref{lip K})$  ($(\ref{Lip H})$ respectively) implies a growth condition on $K(x,t,s)$ (on $H(x,t,\xi)$ respectively), that can be more restrictive than $(\ref{crescita K})$ ($(\ref{crescita H})$ respectively), depending on the value of $\tau$ ($\sigma$ respectively). \\
The model function $a(x,t,\xi)$ which satisfies assumptions
$(\ref{crescita a})$, $(\ref{eqog11})$ or $(\ref{eqog11a})$ is
\[
a(x,t,\xi )=\left\{
\begin{array}{lll}
a(x,t)\left\vert \xi \right\vert ^{p-2}\xi  & \text{if} & 2-\frac{1}{N+1}<p<2,
\\
a(x,t)(1+\left\vert \xi \right\vert ^{2})^{\frac{p-2}{2}}\xi  & \text{if} &
p\geq 2,%
\end{array}%
\right.
\]%
where $
a(x,t)\in L^{\infty }(Q_{T})$ and $a(x,t)>\beta >0.$\\
Examples of functions $K(x,t,s)$ and $H(x,t,\xi)$ are given by
\[
K(x,t,s )=c(x,t)(1+\left\vert s \right\vert )^{\overline{\gamma }%
}
\quad
\text{ with }
 \quad \overline{\gamma }=\min \left\{ \gamma,\tau +1\right\}
\]%
and \[ H(x,t,\xi )=b(x,t)(1+\left\vert \xi \right\vert
)^{\overline{\delta }} \quad \text{ with }\quad \overline{\delta
}=\min \left\{ \delta,\sigma +1\right\},
\]%
where $c(x,t)\in L^{r,1}(Q_{T})$ and  $b(x,t)\in L^{\lambda,1 }(Q_{T})$ with
\[
\left\{
\begin{array}{lll}
r>\frac{p(N+1)-N}{(p-1)(N+1)-N} & \text{if} & 2-\frac{1}{N+1}<p<2, \\
r\geq N+2 & \text{if} & p\geq 2%
\end{array}%
\right.
\]%
and
\[
\left\{
\begin{array}{lll}
\lambda >\frac{p(N+1)-N}{(p-1)(N+1)-N} & \text{if} & 2-\frac{1}{N+1}<p<2, \\
\lambda \geq N+2 & \text{if} & p \geq 2.%
\end{array}%
\right.
\]

\end{remark}


\section{Proof of the results}


This section is devoted to prove Theorems \ref{th1}, \ref{th2} and
\ref{th3}.
We start by a technical lemma which is similar to Lemma 6 of
\cite{blanchard-murat-redvan} for a different parabolic equation with
$L^1$ data. It allows to control the behavior of
some quantities which appear in  the uniqueness process. We stress
that our proof is different to the one in \cite{blanchard-murat-redvan} and
uses only the fact
that two renormalized solutions of \eqref{problema compl} verify
\eqref{12} and \eqref{eqogb} (notice that \eqref{eqogb}  is a consequence
of \eqref{12} and the growth assumption of $K$). See also \cite{blanchard-guibe-redwane} for such a generalization on parabolic equation of the kind
$\frac{\partial b(u)}{\partial
    t}-\dv\left( a(x,t,u,\nabla u)\right)=f+ \dv g.$

\begin{lemma}
\label{Lemma1} Under the assumptions \eqref{ellit}-\eqref{uzero}, let
  $u$ and $v$ be two  renormalized solutions to Problem
\eqref{problema compl}. Let us define for any $0<k<s$
\begin{equation}
  \label{eqog1}
  \begin{split}
    \Gamma&(u,v,s,k) =  \\  &  \int_{\{s-k <  |u| < s+k\}}\Big( a(u,\nabla u)
    \nabla u +
    |K(u)|\, |\nabla u| + |g|^{p'}  \Big) dx dt
    \\ & {} + \int_{\{s-k < |v| < s+k\}} \Big( a(v,\nabla v)\nabla v
     +  |K(v)|\, |\nabla v| + |g|^{p'}\Big) dx dt
  \end{split}
\end{equation}
and for any $0<r<s$
\begin{equation}
  \label{eqog1aa}
  \begin{split}
    \Theta(u,v,s,r)={} &\int_{\{s-r <  |u| < s\}} \Big( a(u,\nabla u)
    \nabla u + |K(u)|\, |\nabla u| \Big)
    dxdt \\
    & + \int_{\{s-r <  |v| < s\}}  \Big( a(v,\nabla v)\nabla v
     + |K(v)|\, |\nabla v|\Big)
    dxdt.
  \end{split}
\end{equation}
Then we have for any $r>0$
\begin{equation}
\liminf_{s\rightarrow \infty} \big( \limsup_{k\rightarrow 0} \frac{1}{k}
\Gamma(u,v,s,k)  + \Theta(u,v,s,r)\big)=0.
\label{eqog2}
\end{equation}
\end{lemma}

\begin{proof}
We argue  by contradiction. Let $r$ be a positive real number.
If the thesis of lemma is not true, let
$\varepsilon_{0}>0$ and let $n_{0}$ be an integer  such that for every real number
$s\geq n_{0}$ we have%
\begin{equation}
 \limsup_{k\rightarrow 0} \frac{1}{k} \Gamma(u,v,s,k) + \Theta(u,v,s,r)
 \geq\varepsilon_{0}.
 \label{eqog0}
\end{equation}

Let us consider the function
\begin{equation*}
  \begin{split}
    F(s)= & \int_{\{ |u|< s\}} \Big( a(u,\nabla u) \nabla u
    +  |K(u)|\, |\nabla u| + |g|^{p'} \Big)    dx dt \\
    & {} +
    \int_{\{ |v| < s\}} \Big( a(v,\nabla v)
    \nabla v + |K(v)|\, |\nabla  v| +     |g|^{p'}\Big)  dxdt.
  \end{split}
\end{equation*}

Due to \eqref{ellit} the function  $F$ is  monotone increasing.
It follows (see e.g. \cite{royden}) that  $F$ is derivable almost
everywhere, with $F'$ measurable,  and that we
have for  any $s>\eta>0$
\begin{equation}
F(s)-F(\eta)\geq\int_{\eta}^{s}F^{\prime}(\xi)d\xi \label{teorema fondamentale}%
\end{equation}
and  for almost any $s>0$
\begin{equation*}
  \begin{split}
    F^{\prime}&(s) =  \\
    & \frac{1}{2}\limsup_{k\rightarrow 0} \frac{1}{k}
\Bigg[    \int_{\{s-k\leq |u|<s+k\}} \Big( a(u,\nabla u) \nabla u
    +  |K(u)|\, |\nabla u| + |g|^{p'} \Big)    dx dt \\
    & {} + \int_{\{ s-k\leq |v| < s+k\}} \Big( a(v,\nabla v) \nabla v
    + |K(v)|\, |\nabla v| + |g|^{p'}\Big) dxdt\Bigg].
  \end{split}
\end{equation*}
Moreover due to \eqref{12} and \eqref{eqogb} and since $g$ belongs to
$(L^{p'}(Q_{T}))^{N}$  we have
\begin{equation}
  \label{eqog00}
  \lim_{s\rightarrow +\infty} \frac{F(s)}{s}=0.
\end{equation}

Due to the definition of $\Gamma(u,v,s,k)$, inequality \eqref{eqog0}
leads to
\[
F'(\xi) + \frac{1}{2}\Theta(u,v,\xi,r) \geq
\frac{\varepsilon_{0}}{2}
\]
for almost $\xi\geq n_{0}$. From  (\ref{teorema fondamentale})  it follows that
\[
\frac{1}{s-n_{0}}\Big( F(s)+ \frac{1}{2} \int_{n_{0}}^{s} \Theta(u,v,\xi,r)
d\xi\Big) \geq \frac{\varepsilon_{0}}{2} +\frac{F(n_{0})}{s-n_{0}}
\quad \text{ for }s> n_{0}.
\]
Due to the definition of $\Theta$, a few computations give
\begin{equation*}
  \begin{split}
 \int_{n_{0}}^{s} \Theta(u,v,\xi,r)
d\xi \leq {} & r \Big(\int_{\{|u|<s\}} |K(u)|\, |\nabla u| dxdt \\
& {} \quad + \int_{\{|v|<s\}} |K(v)|\, |\nabla v| dxdt\Big).
\end{split}
\end{equation*}
It follows that
\begin{multline}
\frac{1}{s-n_{0}}\Big( F(s)+ \frac{r}{2}
\big(\int_{\{|u|<s\}} |K(u)|\, |\nabla u| dxdt \\
 +
\int_{\{|v|<s\}} |K(v)|\, |\nabla v| dxdt\big)\Big)
\geq \frac{\varepsilon_{0}}{2} + \frac{F(n_{0})}{s-n_{0}} \quad \text{ for }s> n_{0}.
\end{multline}
The last inequality contradicts  \eqref{12} and \eqref{eqog00}.
\end{proof}

\begin{proof}[Proof of Theorem \ref{th1}]
The strategy is similar to the proof of Theorem 2 in
\cite{blanchard-murat-redvan}.  It consists to define a smooth
approximation $T_{s}^{\sigma}$
of the truncation $T_{s}$ and to consider two renormalized solutions $u$ and
$v$ to  Problem ($\ref{problema compl}$) for the same data $f,g$ and
$u_{0}$. In Step 1 we plug the test function
$\frac{1}{k}T_{k}\left(  T_{s}^{\sigma
}(u)-T_{s}^{\sigma}(v)\right)  $ in the difference of the equations
(\ref{58}) for $u$ and $v$ in which we have taken
$S=T_{s}^{\sigma}$. This process then leads to equation
\eqref{equazione diff con test}.
In Step 2  we study the behavior of the terms of
\eqref{equazione diff con test} with respect to $\sigma$, $k$ and
$s$, with the help of Lemma \ref{Lemma1}.
In Step 3 we then pass to the limit when
$\sigma\rightarrow0$, $k\rightarrow0$ and $s\rightarrow +\infty$.

\par\medskip

{\slshape Step 1.} Let $u$ and $v$ be two renormalized solutions to Problem
($\ref{problema compl}$) for the same data $f,g$ and $u_{0}.$ For every real
number $s>0$ and $\sigma>0,$ let $T_{s}^{\sigma}$ be the function defined by%

\begin{equation}
\begin{cases}
T_{s}^{\sigma}(0)=0, & \\
\left(  T_{s}^{\sigma}\right)  ^{\prime}(r)=1 &
\text{for }\left\vert r\right\vert <s,\\
\left(  T_{s}^{\sigma}\right)  ^{\prime}(r)=\frac{1}{\sigma}(s+\sigma
-\left\vert r\right\vert ) & \text{for }s\leq\left\vert r\right\vert
\leq s+\sigma,\\
\left(  T_{s}^{\sigma}\right)  ^{\prime}(r)=0 &
\text{for }\left\vert r\right\vert
>s+\sigma.
\end{cases}  \label{test}%
\end{equation}

We take $S=T_{s}^{\sigma}$ in (\ref{58}) for $u$ and $v.$ Subtracting these
two equations and plugging the test function $\frac{1}{k}T_{k}\left(
T_{s}^{\sigma}(u)-T_{s}^{\sigma}(v)\right)  $, we obtain upon
integration on $(0,t)$, that for every
$k>0$, $s>0$, $\sigma>0$,
\begin{multline}
\frac{1}{k}\int_{0}^{t}\left\langle \frac{\partial}{\partial
t}\left[  T_{s}^{\sigma}(u)-T_{s}^{\sigma}(v)\right]  ,T_{k}\left(
T_{s}^{\sigma}(u)-T_{s}^{\sigma}(v)\right)  \right\rangle d{\tau}
\\
  {} + \frac{1}{k} \big(A_{s,k}^{\sigma
}(t) +\widetilde{A}_{s,k}^{\sigma}(t) \big) \label{equazione diff con test}\\
 {} =   \frac{1}{k}\big( C_{s,k}^{\sigma}(t)+\widetilde{C}_{s,k}^{\sigma}(t)+
 F_{s,k}^{\sigma}(t)+G_{s,k}^{\sigma}(t)+\widetilde{G}_{s,k}^{\sigma}(t)\big)
\end{multline}
for almost any $t\in(0,T)$, where $\left\langle \text{ },\text{ }\right\rangle $ denotes the duality
between $L^{1}(\Omega)+ W^{-1,p^{\prime}}(\Omega)$ and
$L^{\infty}(\Omega)\cap W_{0}^{1,p}(\Omega)$ and where
\begin{align*}
A_{s,k}^{\sigma}(t)  = &
\int_{0}^{t}\int_{\Omega}^{}
\left[  \left(  T_{s}^{\sigma}\right)  ^{\prime}(u)a(u,\nabla
u)-\left(  T_{s}^{\sigma}\right)  ^{\prime}(v)a(v,\nabla v)\right] \\
& {} \times \nabla
T_{k}\left(  T_{s}^{\sigma}(u)-T_{s}^{\sigma}(v)\right) dx d\tau,
\end{align*}
\begin{align*}
\widetilde{A}_{s,k}^{\sigma}(t) =  &
\int_{0}^{t}\int_{\Omega}^{} \left(  T_{s}^{\sigma}\right)
^{\prime\prime}(u)a(u,\nabla u)\nabla uT_{k}\left(  T_{s}^{\sigma}%
(u)-T_{s}^{\sigma}(v)\right) dx d\tau \\
& {} -
\int_{0}^{t}\int_{\Omega}^{} \left(  T_{s}^{\sigma}\right)  ^{\prime\prime}(v)a(v,\nabla v)\nabla
vT_{k}\left(  T_{s}^{\sigma}(u)-T_{s}^{\sigma}(v)\right) dxd\tau ,
\end{align*}%
\begin{align*}
C_{s,k}^{\sigma}(t)= &
\int_{0}^{t}\int_{\Omega}^{} \left[  \left(  T_{s}^{\sigma}\right)  ^{\prime
}(u)K(u)-\left(  T_{s}^{\sigma}\right)
^{\prime}(v)K(v)\right]
\\
& \times \nabla T_{k}\left(  T_{s}^{\sigma}(u)-T_{s}^{\sigma}(v)\right) dxd\tau ,
\end{align*}
\begin{align*}
\widetilde{C}_{s,k}^{\sigma}(t)= &
\int_{0}^{t}\int_{\Omega}^{}
\left[  \left(  T_{s}^{\sigma}\right)
^{^{\prime\prime}}(u)K(u)\nabla u-\left(  T_{s}^{\sigma}\right)
^{\prime\prime}(v)K(v)\nabla v\right] \\
& \times T_{k}\left(  T_{s}^{\sigma
}(u)-T_{s}^{\sigma}(v)\right) dxd\tau,
\end{align*}
\begin{gather*}
F_{s,k}^{\sigma}(t)=
\int_{0}^{t}\int_{\Omega}^{} f\left[  \left(  T_{s}^{\sigma}\right)  ^{\prime}(u)-\left(
T_{s}^{\sigma}\right)  ^{\prime}(v)\right]  T_{k}\left(  T_{s}^{\sigma
}(u)-T_{s}^{\sigma}(v)\right) dxd\tau,
\\
G_{s,k}^{\sigma}(t)=
\int_{0}^{t}\int_{\Omega}^{} g\left[  \left(  T_{s}^{\sigma}\right)  ^{\prime}(u)-\left(
T_{s}^{\sigma}\right)  ^{\prime}(v)\right]  \nabla T_{k}\left(  T_{s}^{\sigma
}(u)-T_{s}^{\sigma}(v)\right)dxd\tau,
\\
\widetilde{G}_{s,k}^{\sigma}(t)=
\int_{0}^{t}\int_{\Omega}^{} g\nabla\left[  \left(  T_{s}^{\sigma}\right)
^{\prime}(u)-\left(  T_{s}^{\sigma}\right)  ^{\prime}(v)\right]  T_{k}\left(
T_{s}^{\sigma}(u)-T_{s}^{\sigma}(v)\right)  dxd\tau.
\end{gather*}
In order to pass to the limit in (\ref{equazione diff con test}) when $\sigma
$\ $\rightarrow0,k\rightarrow0$ and $s\rightarrow+\infty,$ we observe that by
(\ref{test}) we have for almost any $t\in(0,T)$
\begin{equation}
T_{s}^{\sigma}(u)\rightarrow T_{s}(u)\text{ in }L^{p}((0,t);W_{0}^{1,p}%
(\Omega))\text{\ and }\text{a.e. in }\Omega\times(0,t) \label{conv
  regolarizzate}%
\end{equation}
and%
\begin{equation}
\left(  T_{s}^{\sigma}\right)  ^{\prime}(u)\rightarrow\chi_{\left\{
\left\vert u\right\vert \leq s\right\}  }\text{ in }L^{q}(\Omega\times
(0,t))\text{ and
a.e. in }\Omega\times(0,t) \label{conv derivata regolarizzata}%
\end{equation}
for every $1<q<+\infty$ for fixed $s>0$ when $\sigma$ tends to zero.

\par
By defining
$\Psi_{k}(r)=\int_{0}^{r} T_{k}(s) ds$, an integration by part (see
\cite{BoccMuratPuel}) gives that for almost any $t\in(0,T)$
\begin{multline}
  \label{eqog2aa}
\frac{1}{k}\int_{0}^{t}\left\langle \frac{\partial}{\partial t}\left[  T_{s}^{\sigma
}(u)-T_{s}^{\sigma}(v)\right] ,  T_{k}\left(  T_{s}^{\sigma}(u)-T_{s}^{\sigma
}(v)\right)  \right\rangle d{\tau} \\
= \frac{1}{k} \int_{\Omega}\Psi_{k}\left(
T_{s}^{\sigma}(u)(t)-T_{s}^{\sigma}(v)(t)\right)  dx.
\end{multline}
We deduce from the above equality that for almost any $t\in (0,T)$
\begin{equation}
  \label{tempo}
  \begin{split}
    \lim_{k\rightarrow
      0}\lim_{\sigma\rightarrow 0}{} & {} \frac{1}{k}\int_{0}^{t}\left\langle
      \frac{\partial}{\partial
        t}\left[ T_{s}^{\sigma }(u)-T_{s}^{\sigma}(v)\right] ,
      T_{k}\left( T_{s}^{\sigma}(u)-T_{s}^{\sigma }(v)\right)
    \right\rangle d{\tau} \\
 &   = \int_{\Omega}|T_{s}(u)(t)-T_{s}(v)(t)| dx.
  \end{split}
\end{equation}

\par\medskip

{\slshape Step 2.}
Reasoning as in \cite{blanchard-murat-redvan} we have
\begin{equation}
\limsup_{k\rightarrow0}\lim_{\sigma\rightarrow 0}
\frac{1}{k}{A_{s,k}^{\sigma}(t)}%
\geq0,\quad\text{ for every }s>0, \label{A}%
\end{equation}%
\begin{equation}
\lim_{s\rightarrow +\infty}
  \lim_{k\rightarrow0}\lim_{\sigma\rightarrow 0}
\frac{1}{k}F_{s,k}^{\sigma}(t) =0, \label{F}%
\end{equation}%
for almost any $t\in(0,T)$.
\par
We give the argument here for completeness. Due to \eqref{conv
  derivata regolarizzata} and \eqref{conv regolarizzate} and
with the help of \eqref{ellit} we have for almost any $t\in(0,T)$
\begin{equation*}
  \begin{split}
  \lim_{\sigma\rightarrow 0} \frac{1}{k} A_{s,k}^{\sigma}(t)= &
  \frac{1}{k} \int_{0}^{t} \int_{\Omega}
  [a(T_{s}(u),DT_{s}(u))-a(T_{s}(v),DT_{s}(v))] \\
  & {} \times
  DT_{k}(T_{s}(u)-T_{s}(v)) dx d\tau
  \end{split}
\end{equation*}
and which can written as
\begin{equation}\label{recall1}
  \begin{split}
  \lim_{\sigma\rightarrow 0} \frac{1}{k} A_{s,k}^{\sigma}(t)= &
  \frac{1}{k} \int_{0}^{t} \int_{\Omega}
  [a(T_{s}(u),DT_{s}(u))-a(T_{s}(u),DT_{s}(v))] \\
  & {} \times
  DT_{k}(T_{s}(u)-T_{s}(v)) dx d\tau
  \\
  & {}
  + \frac{1}{k} \int_{0}^{t} \int_{\Omega}
  [a(T_{s}(u),DT_{s}(v))-a(T_{s}(u),DT_{s}(v))] \\
  & {} \times
  DT_{k}(T_{s}(u)-T_{s}(v)) dx d\tau.
  \end{split}
\end{equation}
Since the operator $a$ is monotone (see \eqref{monotonia}) the first
term of the right hand side of \eqref{recall1} is non negative. It
remains to prove that the second term goes to zero as $k$ goes to
zero. Indeed using the local Lipschitz condition \eqref{loclipa} on
$a$ we get
\begin{equation*}
  \begin{split}
  \frac{1}{k}\bigg| \int_{0}^{t}  & \int_{\Omega}
  [a(T_{s}(u),DT_{s}(v))-a(T_{s}(u),DT_{s}(v))]
   DT_{k}(T_{s}(u)-T_{s}(v)) dx d\tau \bigg|
   \\
  {} \leq {}  & {} \frac{1}{k} \int_{0}^{t}\int_{\Omega}
  \chi_{\{|T_{s}(u)-T_{s}(v)|<k\} }|T_{s}(u)-T_{s}(v)|
   (L_{s}(x,t)+ \gamma_{s}|DT_{s}(v)|^{p-1}) \\
   & {}
   \times |DT_{s}(u)+DT_{s}(v)| dxd\tau
   \\
   \leq {} & \int_{\{0<|T_{s}(u)-T_{s}(v)|<k\} }  (L_{s}(x,t)+
   \gamma_{s}|DT_{s}(v)|^{p-1})
    |DT_{s}(u)+DT_{s}(v)| dxd\tau
\end{split}
\end{equation*}
Due to the regularity of $T_{s}(u)$, $T_{s}(v)$ and $L_{s}$ we have
\[ (L_{s}(x,t)+   \gamma_{s}|DT_{s}(v)|^{p-1})
|DT_{s}(u)+DT_{s}(v)| \in L^{1}(Q_{T}).
\]
Since $\chi_{\{|T_{s}(u)-T_{s}(v)|<k\} }$ tends to zero almost
everywhere in $Q_{T}$ as $k$ goes to zero, the Lebesgue dominated
convergence allows us to conclude that \eqref{A} holds.
\par
As far as \eqref{F} is concerned we have for almost any $t\in(0,T)$
\begin{equation*}
  \lim_{\sigma\rightarrow 0}
\frac{1}{k}F_{s,k}^{\sigma}(t)
=\frac{1}{k}\int_{0}^{t}\int_{\Omega} f (\chi_{\{|u|\leq
  s\}}-\chi_{\{|v|\leq s}) T_{k}(T_{s}(u)-T_{s}(v)) dxd\tau
\end{equation*}
so that for almost any $t\in(0,T)$
\begin{equation*}
\lim_{k\rightarrow 0}  \lim_{\sigma\rightarrow 0}
\frac{1}{k}F_{s,k}^{\sigma}(t)
= \int_{0}^{t}\int_{\Omega} f \times (\chi_{\{|u|\leq  s\}}
-\chi_{\{|v|\leq s}) \sign (u-v) dxd\tau,
\end{equation*}
where $\sign(r)=r/|r|$ for any $r\neq 0$ and $\sign(0)=0$.
Since $u$ and $v$ are finite almost everywhere in $\Omega\times(0,T)$
and since $f$ belongs to $L^{1}(Q_{T})$ the Lebesgue dominated
convergence theorem implies \eqref{F}.

\par\medskip

Now we claim that for almost any $t\in(0,T)$
\begin{equation}
  \label{eqog2a}
  \frac{1}{k}\Big(\big|\widetilde{A}_{s,k}^{\sigma}(t)\big|+
  \big|\widetilde{C}_{s,k}^{\sigma}(t)\big| +
  \big|\widetilde{G}_{s,k}^{\sigma}(t)\big| \Big)
  \leq \frac{M_{1}}{\sigma} \Gamma(u,v,s,\sigma),
\end{equation}
where $M_{1}$ is a constant independent  of $s$, $k$ and $\sigma$ and
where $\Gamma$ is defined in Lemma \ref{Lemma1}.

 Using the definition \eqref{test} of $T_{s}^{\sigma} $, recalling
 that $\nabla u=0$ almost everywhere on $\{ (x,t)\,;\, u(x,t)=r\}$ for
 any $r\in\mathbb{R}$ and since $a(x,t,r,\xi)\xi \geq 0$ we obtain that for any $\sigma$ and any $k>0$
\begin{align}
 \frac{1}{k}\big|\widetilde{A}_{s,k}^{\sigma}(t)\big|
 \leq{} & {} \frac{1}{\sigma}
\bigg[ \int_{0}^{t}\int_{\Omega}
 \chi_{_{\left\{  s < \left\vert
u\right\vert < s+\sigma\right\}  }}a(u,\nabla u)\nabla u dx d\tau
\notag \\
& \qquad + \int_{0}^{t}\int_{\Omega}
 \chi_{_{\left\{  s < \left\vert
v\right\vert < s+\sigma\right\}  }}a (v,\nabla v)\nabla
v dxd\tau\bigg]
\notag \\
  \leq {} & {} \frac{1}{\sigma}
\left[  \int_{_{\left\{  s < \left\vert
u\right\vert < s+\sigma\right\}  }}a(u,\nabla u)\nabla u dx d\tau
+\int_{_{\left\{
s < \left\vert v\right\vert < s+\sigma\right\}  }}a(v,\nabla v)\nabla
v dxd\tau\right].
\label{eqog3a}
\end{align}
Similarly we have for any $\sigma$ and any $k>0$
\begin{multline}
\label{eqog3}
\frac{1}{k}  \big|\widetilde{C}_{s,k}^{\sigma}(t)\big| \leq
\frac{1}{\sigma}\bigg[
\int_{_{\left\{  s < \left\vert
u\right\vert < s+\sigma\right\}  }} |K(u)| \, |\nabla u| dxd\tau
 \\
{}+ \int_{_{\left\{  s < \left\vert
v\right\vert < s+\sigma\right\}  }} |K(v)| \, |\nabla v| dxd\tau
\bigg].
\end{multline}
As far as $\widetilde{G}_{s,k}(t)$ is concerned, we have for any
$\sigma$ and any $k>0$
\begin{align*}
  \frac{1}{k}|\widetilde{G}_{s,k}(t)| & \leq
\frac{1}{\sigma}\bigg[
\int_{_{\left\{  s < \left\vert
u\right\vert < s+\sigma\right\}  }} |g| \, |\nabla u| dxd\tau
+ \int_{_{\left\{  s < \left\vert
v\right\vert < s+\sigma\right\}  }} |g| \, |\nabla v| dxd\tau
\bigg].
\end{align*}
{}From assumption \eqref{ellit} together with Young inequality it
follows that
\begin{equation}
  \label{eqog4}
  \begin{split}
    \frac{1}{k} |\widetilde{G}_{s,k}(t)| \leq & \frac{M_{1}}{\sigma}\bigg( \int_{\{s <
      |u| < s+\sigma\}}\Big( a(u,\nabla u) \nabla u + |g|^{p'} \Big)
    dx d\tau \\ & {} + \int_{\{s < |v| < s+\sigma\}} \Big( a(v,\nabla
    v)\nabla v + |g|^{p'}\Big) dx d\tau\bigg),
  \end{split}
\end{equation}
where $M_{1}$ is a generic constant depending upon $p$ and $\alpha_{0}$.
Estimates \eqref{eqog3a}--\eqref{eqog4} allow us to deduce that
\eqref{eqog2a} holds.

\par\medskip

Now we prove that for almost any $t\in(0,T)$
\begin{equation}
  \label{eqog4a}
\limsup_{\sigma\rightarrow 0}  \frac{1}{k}\Big( \big|{C}_{s,k}^{\sigma}(t)\big| +
  \big|{G}_{s,k}^{\sigma}(t)\big| \Big)
  \leq \frac{M_{1}}{k} \Gamma(u,v,s,k)+\omega(k),
\end{equation}
where $M_{1}$ is a constant independent  of $s$, $k$ and $\sigma$ and
where $\omega$ is a positive function such that $\lim_{k\rightarrow
  0}\omega(k)=0$.

\par
We first write that for almost
any $t\in(0,T)$
\begin{align*}
\limsup_{\sigma\rightarrow0} \frac{1}{k}| C_{s,k}^{\sigma}(t)|   &
{} = \bigg| \frac{1}{k}\int_{0}^{t}
\int_{\Omega}\left[  \chi_{\left\{  \left\vert u\right\vert \leq
s\right\}  }K(u)-\chi_{\left\{  \left\vert v\right\vert \leq s\right\}
}K(v)\right] \\
& \qquad  \times\nabla T_{k}\left(  T_{s}(u)-T_{s}(v)\right)
dxd{\tau} \bigg| \\
&  \leq C_{s,k}^{1}+C_{s,k}^{2}+C_{s,k}^{3},
\end{align*}
where
\[
C_{s,k}^{1}=\frac{1}{k}\int_{Q_{T}}\chi_{\left\{
\left\vert u\right\vert \leq s\wedge|v|>s\right\}  }
|K(u)|\, |\nabla
T_{k}(u-s\sign(v))|dxd{\tau} ,
\]%
\[
C_{s,k}^{2}=\frac{1}{k}\int_{Q_{T}}\chi_{\left\{
\left\vert v\right\vert \leq s\wedge|u|>s\right\}  }|K(v)|\, |\nabla
T_{k}\left(  v-s\sign(u)\right)  |dxd{\tau}
\]
and%
\[
C_{s,k}^{3}=\frac{1}{k}\int_{Q_{T}} \chi_{\left\{
\left\vert v\right\vert \leq s\wedge|u|\leq s\right\}  }%
|K(u)-K(v)|\, |\nabla T_{k}\left(  u-v\right)  |dxd{\tau} .
\]
We estimate $C_{s,k}^{1}$ and $C_{s,k}^{2}.$ By (\ref{crescita K})\ we obtain%
\begin{equation}
\label{eqog5}
\begin{split}
    C_{s,k}^{1} & \leq \frac{1}{k}\int_{Q_{T}}%
    \chi_{\left\{ \left\vert u\right\vert \leq s\wedge|v|>s\right\}}
    \chi_{\{ |u-s\sign(v)|<k\}} |K(u)|\, |\nabla u| dxd{\tau} \\
    & \leq \frac{1}{k}\int_{\left\{ s-k < \left\vert u\right\vert \leq
        s\right\} }|K(u)|\, |\nabla u|dxd\tau
  \end{split}
\end{equation}
and similarly
\begin{equation}
C_{s,k}^{2}\leq\ \frac{1}{k}
\int_{\left\{ s-k < \left\vert v\right\vert \leq    s\right\}  }|K(v)|\, |\nabla
v|dxd\tau.\label{eqog6}
\end{equation}
Finally, since the function $K$ is locally Lipschitz continuous, we have for
some positive $L_{s}$ element of $L^{p'}(Q_{T})$
\begin{align*}
C_{s,k}^{3}  &  =\frac{1}{k}\int_{Q_{T}}
\chi_{\left\{  \left\vert v\right\vert \leq s\wedge|u|\leq s\right\}
}|K(u)-K(v)||\nabla T_{k}\left(  T_{s}(u)-T_{s}(v)\right)  |dxd{\tau} \\
&  \leq\frac{1}{k}\int_{Q_{T}}\chi_{\{0<|T_{s}(v)-T_{s}(u)|<
k\}}L_{s}(x,\tau)\,|T_{s}(u)-T_{s}(v)|
\\
& \qquad\qquad {} \times|\nabla T_{k}\left(  T_{s}(u)-T_{s}(v)\right)
|dx d\tau\\
&  \leq \int_{Q_{T}}\chi_{\{0<|T_{s}(v)-T_{s}(u)| < k\}}%
L_{s}(x,\tau) \, |\nabla T_{k}\left(  T_{s}(u)-T_{s}(v)\right)  |dxd\tau \\
& \leq  \int_{Q_{T}}\chi_{\{0<|T_{s}(v)-T_{s}(u)| < k\}}
L_{s}(x,\tau)\big( |\nabla T_{s}(u)|+ |\nabla T_{s}(v)|\big) dxd\tau.
\end{align*}
Since $L_{s}$ belongs to $L^{p'}(Q_{T})$ and due to
 \eqref{9} the function $L_{s}(x,\tau)\big( |\nabla T_{s}(u)|+ |\nabla
 T_{s}(v)|\big)$ belongs to $L^{1}(Q_{T})$. Because
 $\chi_{\{0<|T_{s}(v)-T_{s}(u)| < k\}}$ tends to $0$ almost everywhere
 in $Q_{T}$ as $k$ goes to $0$ and is bounded by $1$,   the Lebesgue dominated
convergence theorem leads to
\begin{equation}
  \lim_{k\rightarrow 0}
C_{s,k}^{3}=0, \quad \text{for any $s>0$.}
\label{C3}%
\end{equation}
\par
In order to estimate $G_{s,k}(t)$, we obtain for almost any $t\in
(0,T)$
\begin{align*}
\limsup_{\sigma\rightarrow0} \frac{1}{k}|G_{s,k}^{\sigma}(t) |  &
{} = \bigg| \frac{1}{k}\int_{0}^{t}
\int_{\Omega}\left[  \chi_{\left\{  \left\vert u\right\vert \leq
s\right\}  }g-\chi_{\left\{  \left\vert v\right\vert \leq s\right\}
}g\right] \\
& \qquad  \times\nabla T_{k}\left(  T_{s}(u)-T_{s}(v)\right)
dxd{\tau} \bigg| \\
&  \leq G_{s,k}^{1}+G_{s,k}^{2},
\end{align*}
where
\[
G_{s,k}^{1}=\frac{1}{k}\int_{Q_{T}}\chi_{\left\{
\left\vert u\right\vert \leq s\wedge|v|>s\right\}  }
|g|\, |\nabla
T_{k}(u-s\sign(v))|dxd{\tau} ,
\]%
and
\[
G_{s,k}^{2}=\frac{1}{k}\int_{Q_{T}}\chi_{\left\{
\left\vert v\right\vert \leq s\wedge|u|>s\right\}  }|g|\, |\nabla
T_{k}\left(  v-s\sign(u)\right)  |dxd{\tau}.
\]
Since we have
\begin{align*}
G_{s,k}^{1}  %
&  \leq \frac{1}{k}\int_{\left\{ s-k < \left\vert u\right\vert \leq
    s\right\}  }|g|\, |\nabla
u|dxd\tau
\end{align*}
similar arguments to the ones used to deal with $\widetilde{G}_{s,k}$
yield that
\begin{equation}
  \label{eqog7}
  {G}_{s,k}^{1} \leq  \frac{M_{1}}{k} \int_{\{s-k <  |u| < s\}}\Big( a(u,\nabla u)
  \nabla u + |g|^{p'}  \Big) dx d\tau,
\end{equation}
where $M$ is a constant depending upon $p$ and $\alpha_{0}$. With $v$
in place of $u$ in $G_{s,k}^{2}$ we also have
\begin{equation}
  \label{eqog8}
  {G}_{s,k}^{2} \leq  \frac{M_{1}}{k} \int_{\{s-k <  |v| < s\}}\Big( a(v,\nabla v)
  \nabla v + |g|^{p'}  \Big) dx d\tau.
\end{equation}
Estimates  \eqref{eqog5}--\eqref{eqog8} imply  \eqref{eqog4a}

\par\medskip

{\slshape Step 3.}
We are now in a position to prove that $u=v$ almost everywhere in
$Q_{T}$.  Passing to the limit-sup as $\sigma$ goes to $0$ and then to
the limit-sup as $k$ goes to zero in \eqref{equazione diff con test}
with the help of \eqref{tempo}, \eqref{A}, \eqref{F}, \eqref{eqog2a}
and \eqref{eqog4a} leads to for any $s>0$ and for almost any
$t\in(0,T)$
\begin{equation}
  \label{eqog9}
  \begin{split}
    \int_{\Omega}|T_{s}(u)(t)-T_{s}(v)(t)| dx \leq {} & M_{1}
    \limsup_{k\rightarrow 0} \frac{1}{k}\Gamma(u,v,s,k) \\
     & {} + M_{1}
    \limsup_{\sigma\rightarrow 0} \frac{1}{\sigma} \Gamma(u,v,s,\sigma) + \omega(s).
  \end{split}
\end{equation}
Recalling that $u$ (resp. $v$) is finite almost everywhere in $Q_{T}$,
$T_{s}(u)(t)$ (resp. $T_{s}(v)(t)$)  converges almost everywhere to $u(t)$
(resp. $v(t)$) as $s$ goes to infinity for almost any $t\in (0,T)$. By
Fatou lemma we can pass to
the limit-inf as $s$ goes to $+\infty$ in \eqref{eqog9} and we obtain
for almost any $t\in (0,T)$
\begin{equation}
  \label{eqog10}
\int_{\Omega} |u(t)-v(t)| dx \leq 2M_{1} \liminf_{s\rightarrow +\infty}
\limsup_{k\rightarrow 0} \frac{1}{k}\Gamma(u,v,s,k).
\end{equation}
Lemma \ref{Lemma1} allows us to conclude that $\int_{\Omega}
|u(t)-v(t)| dx = 0$ for almost any $t\in(0,T)$ so that  $u=v$
almost everywhere in $Q_{T}.$
\end{proof}

In the case of the complete operator we need the following lemma which
concerns Boccardo-Gallou\"et kind estimates in Lorentz spaces.

\begin{lemma}
\label{lemma tecnico} Assume that $Q_{T}=\Omega\times\left(  0,T\right)  $
with $\Omega$\ open subset of $\mathbb{R}^{N}$\ of finite measure and
$p>1$. Let be $u$ a measurable function satisfying

$T_{k}\left(  u\right)  \in L^{\infty}\big(  (0,T);L^{2}(\Omega)\big)  \cap
L^{p}\big(  (0,T);W_{0}^{1,p}(\Omega)\big)  $ for every $k>0$ and such that
for $\alpha>\frac{2(N+1)}{N+2}$
\begin{equation}
\sup_{t\in\left(  0,T\right)  }\int_{\Omega}\left\vert T_{k}\left(
u(t)\right)  \right\vert ^{2}\leq kM\text{ and}\int_{0}^{T} \int_{\Omega
}\left\vert \nabla T_{k}(u)\right\vert ^{\alpha}\leq C_{0}k^{\frac{\alpha}{2}%
}M^{\frac{\alpha}{2}}, \label{ip lemma}%
\end{equation}
where $M$ and $C_{0}$ are positive constant. Then%
\begin{equation}
\left\Vert u\right\Vert _{L^{\frac{\alpha(N+2)}{2N},\infty}(Q_{T})}\leq CM
\label{3.5b}%
\end{equation}
and%
\begin{equation}
\left\Vert \left\vert \nabla u\right\vert \right\Vert _{L^{\frac{\alpha
(N+2)}{2(N+1)},\infty}(Q_{T})}\leq CM, \label{3.5}%
\end{equation}
where $C$ is a constant depending only on $N$ and $C_{0}$.
\end{lemma}

Such a result being standard we omit the proof of Lemma
\ref{lemma tecnico} (see for example the proof of
Lemma A.1 given in \cite{d-f-g} with a very few modifications).

\begin{proof}
[Proof of Theorem \ref{th2}]
The proof is divided into four steps.
As in the previous theorem we
consider two renormalized solutions $u$ and $v$ of the Problem
($\ref{problema compl}$) for the same data $f,g$ and $u_{0}$. In Step
1, we plug the test function
$T_{k}\left(  T_{s}^{\sigma}(u)-T_{s}^{\sigma}(v)\right)  $ in the
difference of the equations (\ref{58}) for $u$ and $v$ with $S=T_{s}^{\sigma
}$ (defined in (\ref{test})) and we obtain equation \eqref{differenza
  completa}. Step 2 is devoted to estimate the terms of
\eqref{differenza completa}. In Step 3 we pass to the limit as
$\sigma\rightarrow 0$ and $s\rightarrow +\infty$, $k$ being
fixed. Finally in Step 4 using Lemma \ref{lemma tecnico} we give
an estimate of $\nabla u-\nabla v$ in some
suitable Lorentz spaces, which allows us to conclude that $u=v$.

{\slshape Step 1.} Let $u$ and $v$ be two renormalized solutions to Problem
($\ref{problema compl}$) for the same data $f,g$ and $u_{0}.$ For every real
number $s>0$ and $\sigma>0$
we take $S=T_{s}^{\sigma}$ in (\ref{58}) for $u$ and $v.$ Subtracting these
two equations and plugging the test function $\frac{1}{k}T_{k}\left(
T_{s}^{\sigma}(u)-T_{s}^{\sigma}(v)\right)$, we obtain upon
integration on $(0,t)$, that
\begin{multline}
  \label{differenza completa}
\int_{0}^{t}\left\langle \frac{\partial}{\partial
t}\left[  T_{s}^{\sigma}(u)-T_{s}^{\sigma}(v)\right]  ,T_{k}\left(
T_{s}^{\sigma}(u)-T_{s}^{\sigma}(v)\right)  \right\rangle d{\tau}
\\
  {} + A_{s,k}^{\sigma
}(t) +\widetilde{A}_{s,k}^{\sigma}(t)  \\
 {} =B_{s,k}^{\sigma}(t)  +  C_{s,k}^{\sigma}(t)+\widetilde{C}_{s,k}^{\sigma}(t)+
 F_{s,k}^{\sigma}(t)+G_{s,k}^{\sigma}(t)+\widetilde{G}_{s,k}^{\sigma}(t)
\end{multline}
for every
$k>0$, $s>0$, $\sigma>0$ and for almost any $t\in(0,T)$, where
\begin{equation*}
  \begin{split}
    B_{s,k}^{\sigma}(t)= &{}- \int_{0}^{t}\int_{\Omega}
    \big[ (T_{s}^{\sigma}) ^{\prime }(u)H(\nabla
    u)-\left( T_{s}^{\sigma}\right) ^{\prime}(v)H(\nabla
    u)\big] \\
    & \quad \times T_{k}\left( T_{s}^{\sigma}(u)-T_{s}^{\sigma}(v)\right)
    dxd\tau
  \end{split}
\end{equation*}
and the remained terms are defined in the proof of Theorem \ref{th1}.
We now pass to the limit in \eqref{differenza completa} as $\sigma$
goes to zero and then as $s$ goes to $+\infty$.
\par\medskip
{\slshape Step 2.}
We recall that for almost any $t\in(0,T)$
\begin{multline*}
\int_{0}^{t}\left\langle \frac{\partial}{\partial t}\left[  T_{s}^{\sigma
}(u)-T_{s}^{\sigma}(v)\right] ,  T_{k}\left(  T_{s}^{\sigma}(u)-T_{s}^{\sigma
}(v)\right)  \right\rangle d{\tau} \\
=  \int_{\Omega}\Psi_{k}\left(
T_{s}^{\sigma}(u)(t)-T_{s}^{\sigma}(v)(t)\right)  dx.
\end{multline*}
Due to the definition of $T_{s}^{\sigma}$ we obtain that
\begin{equation*}
  \begin{split}
    \lim_{\sigma\rightarrow 0}\int_{0}^{t}\Big\langle \frac{\partial}{\partial t}\left[  T_{s}^{\sigma
}(u)-T_{s}^{\sigma}(v)\right] , {}& {}  T_{k}\left(  T_{s}^{\sigma}(u)-T_{s}^{\sigma
}(v)\right)  \Big\rangle d{\tau} \\
& =  \int_{\Omega}\Psi_{k}\left(
T_{s}(u)(t)-T_{s}(v)(t)\right)  dx
\end{split}
\end{equation*}
and since $u$ and $v$ are finite almost everywhere in $Q_{T}$, from Fatou
lemma it follows that
\begin{equation}
  \label{eqog12}
  \begin{split}
    \liminf_{s\rightarrow +\infty}\lim_{\sigma\rightarrow 0}\int_{0}^{t}\Big\langle \frac{\partial}{\partial t}\left[  T_{s}^{\sigma
}(u)-T_{s}^{\sigma}(v)\right] , {}& {}  T_{k}\left(  T_{s}^{\sigma}(u)-T_{s}^{\sigma
}(v)\right)  \Big\rangle d{\tau} \\
& \geq   \int_{\Omega}\Psi_{k}\left(
u(t)-v(t)\right)  dx.
\end{split}
\end{equation}
Since $H(\nabla u)$ and $H(\nabla v)$ belong to $L^{1}(Q_{T})$ and
since $u$ and $v$ are finite almost everywhere in $Q_{T}$, the
Lebesgue theorem yields that
\begin{equation*}
\begin{aligned}
  \lim_{s\rightarrow +\infty}\lim_{\sigma\rightarrow0} B_{s,k}^{\sigma}(t)
 & = -\int_{0}^{t} \int_{\Omega} \left[ H(\nabla u)-H(\nabla
   v)\right] T_{k}(u-v) dxd\tau.
\end{aligned}
\end{equation*}
Using the Lipschitz condition \eqref{Lip H} on $H$ and \eqref{ip 1}
we obtain
\begin{multline*}
     \int_{0}^{t}
\int_{\Omega} \big|\left[H(\nabla u)-H(\nabla
   v)\right] T_{k}(u-v) \big| dxd\tau \leq  \\
 k \left\Vert b\right\Vert _{L^{\lambda
,1}(Q_{t})}\left\Vert 1+\left\vert \nabla u\right\vert +\left\vert \nabla
v\right\vert \right\Vert _{L^{q,\infty}(Q_{t})}^{\sigma}\left\Vert
\left\vert \nabla u-\nabla v\right\vert \right\Vert _{L^{\theta,\infty
}(Q_{t})}
\end{multline*}
with
\begin{gather*}
  \frac{1}{\lambda}+\frac{\sigma}{q}+\frac{1}{\theta}=1,
  \quad
1\leq q\leq\frac{N(p-1)+p}{N+1},\quad
\theta=\frac{N+2}{N+1}\quad \text{and} \quad
\lambda\geq N+2.
\end{gather*}
It follows that for almost any $t\in(0,T)$
\begin{multline}
  \label{eqog14}
   \lim_{s\rightarrow +\infty}\lim_{\sigma\rightarrow0}
  | B_{s,k}^{\sigma}(t) | \\
  \leq  k \left\Vert b\right\Vert _{L^{\lambda
,1}(Q_{t})}\left\Vert 1+\left\vert \nabla u\right\vert +\left\vert \nabla
v\right\vert \right\Vert _{L^{q,\infty}(Q_{t})}^{\sigma}\left\Vert
\left\vert \nabla u-\nabla v\right\vert \right\Vert _{L^{\theta,\infty
}(Q_{t})}.
\end{multline}
Since $f$ belongs to $L^{1}(Q_{T})$ while $u$ and $v$ are finite
almost everywhere in $Q_{T}$ we have
\begin{equation}
  \label{eqog16}
  \begin{split}
  \lim_{s\rightarrow+\infty}  \lim_{\sigma\rightarrow0} F_{s,k}^{\sigma}(t)
  = {} &
  \lim_{s\rightarrow+\infty}\int_{0}^{t}\int_{\Omega}f\left[  \chi_{\left\{  \left\vert
          u\right\vert \leq s\right\}  }-\chi_{\left\{  \left\vert
          v\right\vert  \leq s\right\}  }\right] \\
   & \qquad\qquad\qquad {} \times  T_{k}\left(  T_{s}(u)-T_{s}(v)\right)  dxd\tau=0
\end{split}
\end{equation}
We now deal with $A_{s,k}^{\sigma}$, $C_{s,k}^{\sigma}$ and
$G_{s,k}^{\sigma}$. From the definition of $T_{s}^{\sigma}$ and
\eqref{eqog11} we get
\begin{equation}
  \label{eqog13}
\begin{split}
   \lim_{\sigma\rightarrow0} A_{s,k}^{\sigma} (t)
 = {}& \begin{aligned}[t]\int_{0}^{t} \int_{\Omega} \left[  \chi_{\left\{  \left\vert
u\right\vert \leq s\right\}  }a(\nabla u)-\chi_{\left\{  \left\vert v\right\vert
\leq s\right\}  }a(\nabla v)\right]  \\
 \times \nabla T_{k}\left(  T_{s}(u)-T_{s}%
(v)\right)  dxd\tau
\end{aligned}
\\
 = {} &
 \begin{aligned}[t]
   \int_{0}^{t} \int_{\Omega} \chi_{\{T_{s}(u)-T_{s}(v)|<k\}}
   \big[a(\nabla T_{s}(u))-a(\nabla T_{s}(v))\big]
   \\
    \times \big(\nabla T_{s}(u)-\nabla T_{s}(v)\big) dx d\tau
 \end{aligned}
 \\
 \geq {}&
 \begin{aligned}[t]
   \beta \int_{0}^{t} \int_{\Omega} \chi_{\{T_{s}(u)-T_{s}(v)|<k\}}
   (1+|\nabla T_{s}(u)|+|\nabla T_{s}(v)|)^{p-2} \\
    \times  |\nabla    T_{s}(u)-\nabla T_{s}(v)|^{2} dx d\tau.
 \end{aligned}
\end{split}%
\end{equation}
Since $u$ and $v$ are finite almost everywhere Fatou lemma then
implies
 \begin{equation}
   \label{eqog13b}
   \begin{split}
     \liminf_{s\rightarrow +\infty} \lim_{\sigma\rightarrow0}
     A_{s,k}^{\sigma} \geq {} & {} \beta\int_{0}^{t} \int_{\Omega}
     \chi_{\{|u-v|<k\}}
     (1+|\nabla u|+|\nabla (v)|)^{p-2} \\
     & {} \times |\nabla u-\nabla v|^{2} dx d\tau.
   \end{split}
 \end{equation}
Using assumption \eqref{lip K} we have
\begin{equation}
\begin{split}
  \lim_{\sigma\rightarrow0} |C_{s,k}^{\sigma}(t)|
  \leq &
  \begin{aligned}[t]
    \int_{0}^{t}\int_{\Omega}  \left|  \chi_{\left\{  \left\vert u\right\vert
\leq s\right\}  }K(u)-\chi_{\left\{  \left\vert v\right\vert
\leq s\right\}
}K(v)\right| \\
 \times \big|\nabla T_{k}\left(  T_{s}(u)-T_{s}(v)\right)\big|  dxd\tau
 \end{aligned}
 \\
 \leq &
 \begin{aligned}[t]
   \int_{0}^{t}\int_{\Omega}  \chi_{\{|u|\leq s \wedge|v| \leq s\}} c(x,\tau)
   \big(1 + |u|+|v|\big)^{\tau} \\
   \times |u-v| \, |\nabla T_{k}   ( u-v) |
   dx d\tau
 \end{aligned}
 \\
 &{} +  \int_{0}^{t}\int_{\Omega} \chi_{\{s-k<|v|\leq s\}}
|K(v)|\, |\nabla v|dx d\tau
 \\
  &{} +  \int_{0}^{t}\int_{\Omega} \chi_{\{s-k<|u|\leq s\}}
 |K(u)|\, |\nabla u|dx d\tau.
\end{split}
\label{eqog15}
\end{equation}
{}From H\"older inequality and condition \eqref{ip 1} we obtain
\begin{multline}
  \int_{0}^{t}\int_{\Omega}  c(x,\tau)
  \big(1 + |u|+|v|\big)^{\tau}
  |\nabla   u-\nabla v| dxd\tau
  \\
  \leq  \left\Vert c\right\Vert _{L^{r,1}(Q_{t})}\left\Vert
1+\left\vert u\right\vert +\left\vert v\right\vert \right\Vert _{L^{\overline
{q},\infty}(Q_{t})}^{\tau}\left\Vert \left\vert \nabla u-\nabla
v\right\vert \right\Vert _{L^{\theta,\infty}(Q_{t})}
\end{multline}
with
\begin{gather*}
\frac{1}{r}+\frac{\tau}{\overline{q}}+\frac{1}{\theta}=1, \quad
1\leq\overline{q}\leq\frac{N(p-1)+p}{N},\quad
\theta=\frac{N+2}{N+1} \quad \text{and}\quad
r>\frac{N+p}{p-1}.
\end{gather*}
{}From the regularities of $c$, $u$, $v$, $\nabla u$ and $\nabla v$ it
follows that $ c(x,\tau)  \big(1 + |u|+|v|\big)^{\tau}   |\nabla
u-\nabla v|$ belongs to $L^{1}(Q_{t})$ for any $t\in(0,T)$.
Recalling the definition \eqref{eqog1aa} of $\Theta$ in Lemma \ref{Lemma1}
leads to
\begin{equation}
  \label{eqog15a}
  \begin{split}
   \lim_{\sigma\rightarrow0} |C_{s,k}^{\sigma}(t)|
  \leq  & k\left\Vert c\right\Vert _{L^{r,1}(Q_{t})}\left\Vert
1+\left\vert u\right\vert +\left\vert v\right\vert \right\Vert _{L^{\overline
{q},\infty}(Q_{t})}^{\tau} \\
& \quad {} \times \left\Vert \left\vert \nabla u-\nabla
v\right\vert \right\Vert _{L^{\theta,\infty}(Q_{t})}
 + \Theta(u,v,s,k)
\end{split}
\end{equation}
for any $k>0$.
\par
We now study $G_{s,k}^{\sigma}(t)$. We first have
\begin{equation*}
  \lim_{\sigma\rightarrow0}
G_{s,k}^{\sigma}(t)   =\int_{0}^{t}\int_{\Omega}g\left[  \chi_{\left\{  \left\vert u\right\vert
<s\right\}  }-\chi_{\left\{  \left\vert v\right\vert <s\right\}  }\right]
\nabla T_{k}\left(  T_{s}(u)-T_{s}(v)\right)  dxdt.
\end{equation*}
It follows that
\begin{equation*}
\begin{split}
  \lim_{\sigma\rightarrow0}
|G_{s,k}^{\sigma}(t) |  \leq & \int_{0}^{t}\int_{\Omega}\chi_{\{s-k<|u|<s\}}|g|
|\nabla u| dx d\tau
\\
& {} + \int_{0}^{t}\int_{\Omega}\chi_{\{s-k<|v|<s\}}|g|
|\nabla v| dx d\tau.
\end{split}
\end{equation*}
With Young inequality and integrating on $Q_{T}$ in place of
$\Omega\times (0,t)$ we obtain
\begin{equation*}
\begin{split}
  \lim_{\sigma\rightarrow0}
|G_{s,k}^{\sigma}(t) |  \leq & \frac{1}{p'}\int_{Q_T}
\big(\chi_{\{s-k<|u|<s\}}+\chi_{\{s-k<|v|<s\}}\big) |g|^{p'} dx d\tau
\\
& {} + \frac{1}{p}\int_{\{s-k<|u|<s\}}|\nabla u|^{p} dx d\tau +
\frac{1}{p}\int_{\{s-k<|v|<s\}}|\nabla v|^{p} dx d\tau .
\end{split}
\end{equation*}
Since $u$ and $v$ are finite almost everywhere in $Q_{T}$ the function
$\big(\chi_{\{s-k<|u|<s\}}+\chi_{\{s-k<|v|<s\}}\big) |g|^{p'}$
converges to zero as $s$ goes to $+\infty$ in $L^{1}(Q_{T})$. Since the
operator $a$ is elliptic (see assumption \eqref{ellit}) and recalling
the definition of $\Theta$ in Lemma \ref{Lemma1} we then obtain
\begin{equation}
  \label{eqog17}
  \lim_{\sigma\rightarrow0}
|G_{s,k}^{\sigma}(t) |\leq \frac{1}{\alpha_{0}} \Theta(u,v,s,k) + \omega(s)
\end{equation}
where $\omega(s)$ is a generic function which converges to $0$ as $s$
goes to infinity.
\par
We recall (see
 \eqref{eqog2a} in the proof of Theorem \ref{th1}) that
for almost any $t\in(0,T)$
\begin{equation}\label{eqog12a}
\big|\widetilde{A}_{s,k}^{\sigma}(t)\big|+
  \big|\widetilde{C}_{s,k}^{\sigma}(t)\big| +
  \big|\widetilde{G}_{s,k}^{\sigma}(t)\big|
  \leq \frac{M_{1}k}{\sigma} \Gamma(u,v,s,\sigma).
\end{equation}
{\relax}From  estimates \eqref{eqog15a}, \eqref{eqog17} and \eqref{eqog12a}
it follows that
\begin{equation*}
  \begin{split}
   \limsup_{\sigma\rightarrow0} \big(&|C_{s,k}^{\sigma}(t)|
   +
   |\widetilde{A}_{s,k}^{\sigma}(t)|+
|G_{s,k}^{\sigma}(t) |+
  |\widetilde{C}_{s,k}^{\sigma}(t)| +
  |\widetilde{G}_{s,k}^{\sigma}(t)|\big) \\
  \leq {} &{} k\left\Vert c\right\Vert _{L^{r,1}(Q_{t})}\left\Vert
1+\left\vert u\right\vert +\left\vert v\right\vert \right\Vert _{L^{\overline
{q},\infty}(Q_{t})}^{\tau}
 \left\Vert \left\vert \nabla u-\nabla
v\right\vert \right\Vert _{L^{\theta,\infty}(Q_{t})}
\\ &{}  + \Theta(u,v,s,k) + \frac{1}{\alpha_{0}} \Theta(u,v,s,k) +
\omega(s) \\
& {} + M_{1} k \limsup_{\sigma\rightarrow 0} \frac{1}{\sigma} \Gamma(u,v,s,\sigma).
 \end{split}
\end{equation*}
By the above
inequality and Lemma \ref{Lemma1} we can conclude
that for almost any $t\in(0,T)$
\begin{equation}
  \label{eqog12b}
   \begin{split}
   \liminf_{s\rightarrow+\infty}\limsup_{\sigma\rightarrow0} \big(&|C_{s,k}^{\sigma}(t)|
   +
   |\widetilde{A}_{s,k}^{\sigma}(t)|+
|G_{s,k}^{\sigma}(t) |+
  |\widetilde{C}_{s,k}^{\sigma}(t)| +
  |\widetilde{G}_{s,k}^{\sigma}(t)|\big) \\
  \leq {} &{} k\left\Vert c\right\Vert _{L^{r,1}(Q_{t})}\left\Vert
1+\left\vert u\right\vert +\left\vert v\right\vert \right\Vert _{L^{\overline
{q},\infty}(Q_{t})}^{\tau}
 \left\Vert \left\vert \nabla u-\nabla
v\right\vert \right\Vert_{L^{\theta,\infty}(Q_{t})} .
 \end{split}
\end{equation}

\par
{\slshape Step 3.}
We are now able to pass to the limit in \eqref{differenza
  completa}. Indeed gathering
 \eqref{eqog12}, \eqref{eqog14}, \eqref{eqog16}, \eqref{eqog13b}
and \eqref{eqog12b}, we get
\begin{align*}
  \int_{\Omega}\Psi_{k}&\left(  u(t)-v(t)\right)  dx \\
 & +
\frac{\beta}{2}\int_{0}^{t} \int_{\Omega}
     \chi_{\{|u-v|<k\}}
     (1+|\nabla u|+|\nabla (v)|)^{p-2}
      |\nabla u-\nabla v|^{2} dx d\tau
\\
\leq {} &   k\left\Vert c\right\Vert _{L^{r,1}(Q_{t})}\left\Vert
1+\left\vert u\right\vert +\left\vert v\right\vert \right\Vert _{L^{\overline
{q},\infty}(Q_{t})}^{\tau} \left\Vert \left\vert \nabla u-\nabla
v\right\vert \right\Vert _{L^{\theta,\infty}(Q_{t})}
\\
&  +  k \left\Vert b\right\Vert _{L^{\lambda
,1}(Q_{t})}\left\Vert 1+\left\vert \nabla u\right\vert +\left\vert \nabla
v\right\vert \right\Vert _{L^{q,\infty}(Q_{t})}^{\sigma}\left\Vert
\left\vert \nabla u-\nabla v\right\vert \right\Vert _{L^{\theta,\infty
}(Q_{t})}
\end{align*}
 for almost any $t\in(0,T)$.
It is worth noting that the above inequality implies
\[
 \chi_{\{|u-v|<k\}} (1+|\nabla
u|+|\nabla (v)|)^{p-2}   |\nabla u-\nabla v|^{2}\in L^{1}(Q_{T}).
\]
Since  $(1+|\xi|+|\xi'|)^{p-2}   |\xi-\xi'|^{2}\geq  |\xi-\xi'|^{2}$
for any $\xi$, $\xi'$ in $\mathbb{R}^{N}$,  we obtain that
$T_{k}(u-v)$ belongs to
$L^{2}((0,T); H^{1}_{0}(\Omega))$.

Due to the definition of $\Psi_{k}$, taking the supremum for $t\in\left(
0,t_{1}\right)  $, where $t_{1}\in(0,T)$ will be chosen later, leads to
\begin{equation}
\label{kM}
  \frac{1}{2}\underset{t\in(0,t_{1})}{\sup}\int_{\Omega}\left\vert
T_{k}\left(  u-v\right)  \right\vert ^{2}dx +
\frac{\beta}{2}\int_{0}^{t_{1}} \int_{\Omega}
          |\nabla T_{k}(u- v)|^{2} dx d\tau \leq kM
\end{equation}
where%
\begin{equation}
\label{M}%
\begin{split}
M= & \left\Vert b\right\Vert _{L^{\lambda
,1}(Q_{t_{1}})}\left\Vert 1+\left\vert \nabla u\right\vert +\left\vert \nabla
v\right\vert \right\Vert _{L^{q,\infty}(Q_{t_{1}})}^{\sigma}\left\Vert
\left\vert \nabla u-\nabla v\right\vert \right\Vert _{L^{\theta,\infty
}(Q_{t_{1}})} \\
&  +   \left\Vert c\right\Vert _{L^{r,1}(Q_{t_{1}})}\left\Vert
1+\left\vert u\right\vert +\left\vert v\right\vert \right\Vert _{L^{\overline
{q},\infty}(Q_{t_{1}})}^{\tau}\left\Vert \left\vert \nabla u-\nabla
v\right\vert \right\Vert _{L^{\theta,\infty}(Q_{t_{1}})}.
\end{split}
\end{equation}
By (\ref{kM}) and Lemma \ref{lemma tecnico} we get%
\begin{equation}
\left\Vert \nabla u-\nabla v\right\Vert _{L^{\theta,\infty}(Q_{t_{1}})}\leq
CM\text{ }\label{teta}%
\end{equation}
for some constant $C>0$ independent on $u$ and $v$ and $\theta=\frac{N+2}%
{N+1}$.
\par\smallskip
{\slshape Step 4.}
Using (\ref{M}) and (\ref{teta}) we obtain%
\begin{align}
\left\Vert \left\vert \nabla u-\nabla v\right\vert \right\Vert _{L^{\theta
,\infty}(Q_{t_{1}})} \leq {} &  C\left[  \left\Vert b\right\Vert _{L^{\lambda
,1}(Q_{t_{1}})}\left\Vert 1+\left\vert \nabla u\right\vert +\left\vert \nabla
v\right\vert \right\Vert _{L^{q,\infty}(Q_{t_{1}})}^{\sigma}\right.  \label{after lemma}\\
& \quad \left. + \left\Vert c\right\Vert _{L^{r,1}(Q_{t_{1}})}\left\Vert
1+\left\vert u\right\vert +\left\vert v\right\vert \right\Vert _{L^{\overline
{q},\infty}(Q_{t_{1}})}^{\tau} \right] \nonumber \\
& {} \times \left\Vert \left\vert \nabla u-\nabla
v\right\vert \right\Vert _{L^{\theta,\infty}(Q_{t_{1}})}. \nonumber
\end{align}
Since $c$ belongs to $L^{r,1}(Q_{T})$ and since $b$ belongs to
$L^{\lambda,1}(Q_{T})$, choosing $t_{1}$ small enough such that
\begin{multline}
1-C\Big(  \left\Vert b\right\Vert _{L^{\lambda,1}(Q_{t_{1}})}\left\Vert
1+\left\vert \nabla u\right\vert +\left\vert \nabla v\right\vert \right\Vert
_{L^{q,\infty}(Q_{t_{1}})}^{\sigma}
\\
+\left\Vert c\right\Vert _{L^{r,1}%
(Q_{t_{1}})}\left\Vert 1+\left\vert u\right\vert +\left\vert v\right\vert
\right\Vert _{L^{\overline{q},\infty}(Q_{t_{1}})}^{\tau}\Big)  >0,
\label{prima t1}%
\end{multline}
then  (\ref{after lemma}) gives
\begin{equation}
\left\Vert \left\vert \nabla u-\nabla v\right\vert \right\Vert _{L^{\theta
,\infty}(Q_{t_{1}})}\leq0 \label{primo}%
\end{equation}
with $\theta=\frac{N+2}{N+1}$.
\par

Now we use the same technique as in \cite{porzio} (see also
\cite{d-f-g}). We consider a partition of
the entire interval $\left[  0,T\right]  $ into a finite number of intervals
$\left[  0,t_{1}\right]  $, $\left[  t_{1},t_{2}\right]  ,...$, $\left[
t_{n-1,}T\right]  $ such that for each interval $\left[  t_{i-1},t_{i}\right]
$ a similar condition to (\ref{prima t1}) holds. In this way in each cylinder
$Q_{t_{i}}=\Omega\times\left[  t_{i-1},t_{i}\right]  $ we obtain estimates of
type (\ref{primo}). Then we can deduce that%
\[
\left\Vert \left\vert \nabla u-\nabla v\right\vert \right\Vert _{L^{\theta
,\infty}(Q_{T})}\leq0\text{ \ \ for some }\theta\geq1,\text{ }%
\]
that implies that $u=v$ a.e. in $Q_{T}.$
\end{proof}

\begin{proof}[Proof of Theorem \ref{th3}]

  The strategy of the proof is the same as in Theorem \ref{th2}
  and relies on
  passing to the limit in \eqref{differenza completa}.
  The main
  differences are in dealing the terms $A_{s,k}^{\sigma}(t)$,
  $B_{s,k}^{\sigma}$ and $C_{s,k}^{\sigma}(t)$ and
  the estimate on $\nabla T_{k}(u-v)$. We recall
  \eqref{differenza completa}:
\begin{multline*}
\int_{0}^{t}\left\langle \frac{\partial}{\partial
t}\left[  T_{s}^{\sigma}(u)-T_{s}^{\sigma}(v)\right]  ,T_{k}\left(
T_{s}^{\sigma}(u)-T_{s}^{\sigma}(v)\right)  \right\rangle d{\tau}
\\
  {} + A_{s,k}^{\sigma
}(t) +\widetilde{A}_{s,k}^{\sigma}(t)  \\
 {} =B_{s,k}^{\sigma} (t) +  C_{s,k}^{\sigma}(t)+\widetilde{C}_{s,k}^{\sigma}(t)+
 F_{s,k}^{\sigma}(t)+G_{s,k}^{\sigma}(t)+\widetilde{G}_{s,k}^{\sigma}(t)
\end{multline*}
for any $s>0$, any $k>0$ and
  any $\sigma>0$ and for almost any $t\in(0,T)$.
 Reasoning as in the previous theorem by assumption \eqref{eqog11a}, we
  obtain that
   \begin{equation}
   \label{eqog13bb}
   \begin{split}
     \liminf_{s\rightarrow +\infty} \lim_{\sigma\rightarrow0}
     A_{s,k}^{\sigma}(t) \geq {} & {} \beta\int_{0}^{t} \int_{\Omega}
     \chi_{\{|u-v|<k\}}
    \frac{\left\vert \nabla u- \nabla v)\right\vert ^{2}}{\left(  \left\vert \nabla
u\right\vert +\left\vert \nabla v\right\vert \right)  ^{2-p}}
 dx d\tau.
   \end{split}
 \end{equation}
 As far as $B_{s,k}^{\sigma}(t)$ is concerned, a few computations,
 estimates \eqref{stimagradiente} and \eqref{stimau}, condition
 \eqref{ip 2} and  H\"older inequality lead to
 \begin{multline*}
     \int_{0}^{t}
\int_{\Omega} \big|\left[H(\nabla u)-H(\nabla
   v)\right] T_{k}(u-v) \big| dxd\tau \leq  \\
 k  \left\Vert b\right\Vert _{L^{\lambda
,1}(Q_{t})}\left\Vert 1+\left\vert \nabla u\right\vert +\left\vert \nabla
v\right\vert \right\Vert _{L^{q,\infty}(Q_{t})}^{\sigma}\left\Vert
\left\vert \nabla u-\nabla v\right\vert \right\Vert _{L^{\theta,\infty
}(Q_{t})}
 \end{multline*}
 with
 \begin{gather*}
   \frac{1}{\lambda}+\frac{\sigma}{q}+\frac{1}{\theta}=1,\quad
1\leq q\leq\frac{N(p-1)+p}{N+1}, \\
\theta=\frac{\alpha(N+2)}{2(N+1)},\quad
\lambda\geq N+2 \quad \text{and} \quad
\alpha<\frac{2p(N+1)-2N}{N+2}.%
 \end{gather*}
 Similarly we obtain
 \begin{equation}
  \label{eqog20}
  \begin{split}
\liminf_{s\rightarrow +\infty}
   \lim_{\sigma\rightarrow0} |C_{s,k}^{\sigma}(t)|
  \leq {} & k
   \left\Vert c\right\Vert _{L^{r,1}(Q_{t})}\left\Vert
1+\left\vert u\right\vert +\left\vert v\right\vert \right\Vert _{L^{\overline
{q},\infty}(Q_{t})}^{\tau}\\
&{}  \times \left\Vert \left\vert \nabla u-\nabla
v\right\vert \right\Vert _{L^{\theta,\infty}(Q_{t})}
\end{split}
\end{equation}
 with
 \begin{gather*}
   \frac{1}{r}+\frac{\tau}{\overline{q}}+\frac{1}{\theta}=1,\quad
1\leq\overline{q}\leq\frac{N(p-1)+p}{N},\\
\theta=\frac{\alpha(N+2)}{2(N+1)},\quad
r>\frac{N+p}{p-1},\quad \text{and} \quad
\alpha<\frac{2p(N+1)-2N}{N+2}.%
 \end{gather*}

Then the analogous of \eqref{kM} is
 \begin{equation}\label{dopolimite p<2}
  \frac{1}{2}\underset{t\in(0,t_{1})}{\sup}\int_{\Omega}\left\vert
T_{k}\left(  u-v\right)  \right\vert ^{2}dx+\beta\int_{0}^{t_{1}}\int_{\Omega
}\frac{\left\vert \nabla T_{k}(u-v)\right\vert ^{2}}{\left(  \left\vert \nabla
u\right\vert +\left\vert \nabla v\right\vert \right)  ^{2-p}}%
dxdt  \leq k M
\end{equation}
where $t_{1}$ will be chosen later and $M$ is defined in the proof of Theorem \ref{th2} (see \eqref{M}).
Then we obtain that
\begin{equation}
\beta \int_{0}^{t_{1}}\int_{\Omega}\frac{\left\vert \nabla T_{k}(u-v)\right\vert
^{2}}{\left(  \left\vert \nabla u\right\vert +\left\vert \nabla v\right\vert
\right)  ^{2-p}}dxdt\leq Mk,\label{ip lemma parziale}%
\end{equation}%
\begin{equation}
\frac{1}{2}\underset{t\in(0,t_{1})}{\sup}\int_{\Omega}\left\vert T_{k}\left(
u-v\right)  \right\vert ^{2}\leq Mk,\label{ip lemma 1}.%
\end{equation}
If $1\leq\alpha<p$, by H\"{o}lder
inequality and (\ref{ip lemma parziale}) we have
\begin{equation}
  \label{prima lemma}
  \begin{split}
\int_{0}^{t_{1}}\int_{\Omega}  \vert \nabla & T_{k}(u-v)\vert ^{\alpha
}dxdt
\\
& {}=\int_{0}^{t_{1}}\int_{\Omega}\left\vert \nabla T_{k}(u-v)\right\vert
^{\alpha}\frac{\left(  \left\vert \nabla u\right\vert +\left\vert \nabla
v\right\vert \right)  ^{(2-p)\frac{\alpha}{2}}}{\left(  \left\vert \nabla
u\right\vert +\left\vert \nabla v\right\vert \right)  ^{(2-p)\frac{\alpha}{2}%
}}dxdt
\\
& {}\leq \left(  \int_{0}^{t_{1}}\int_{\Omega}\frac{\left\vert \nabla T_{k}%
(u-v)\right\vert ^{2}}{\left(  \left\vert \nabla u\right\vert +\left\vert
\nabla v\right\vert \right)  ^{2-p}}dxdt\right)^{\frac{\alpha}{2}}
\\
& \qquad{} \times \left(
\int_{0}^{t_{1}}\int_{\Omega}\left(  \left\vert \nabla u\right\vert
+\left\vert \nabla v\right\vert \right)  ^{\frac{(2-p)\alpha}{2-\alpha}%
}dxdt\right)  ^{\frac{2-\alpha}{2}}\\
& {} \leq\left(  Mk\right)  ^{\frac{\alpha}{2}}\left(  \int_{0}^{t_{1}}%
\int_{\Omega}\left(  \left\vert \nabla u\right\vert +\left\vert \nabla
v\right\vert \right)  ^{\frac{(2-p)\alpha}{2-\alpha}}dxdt\right)
^{\frac{2-\alpha}{2}}.
\end{split}
\end{equation}
By \eqref{stimagradiente} the last integral in
(\ref{prima lemma}) is finite if
\begin{equation}
\alpha<\frac{2p(N+1)-2N}{N+2}.\label{alfa}%
\end{equation}
We observe that condition (\ref{alfa}) and the conditon on $\alpha$ in Lemma
\ref{lemma tecnico} are compatible only if $p>2-\frac{1}{N+1}.$ Then by
(\ref{prima lemma}),(\ref{alfa}) and by H\"{o}lder inequality we have
\begin{equation}
\int_{0}^{t_{1}}\int_{\Omega}\left\vert \nabla T_{k}(u-v)\right\vert ^{\alpha
}dxdt\leq C\left(  Mk\right)  ^{\frac{\alpha}{2}},\label{2 ipotesi lemma}%
\end{equation}
where $C$ is constant independent on $t_{1}$.
\par
By (\ref{ip lemma 1}), (\ref{2 ipotesi lemma}) and the definition of $M$ Lemma \ref{lemma tecnico}
 gives for  $\theta=\frac{\alpha(N+2)}{2(N+1)}$
\begin{align}
\left\Vert \left\vert \nabla u-\nabla v\right\vert \right\Vert _{L^{\theta
,\infty}(Q_{t_{1}})}\leq {} & C\Big[  \left\Vert b\right\Vert _{L^{\lambda
,1}(Q_{t_{1}})}\left\Vert 1+\left\vert \nabla u\right\vert +\left\vert \nabla
v\right\vert \right\Vert _{L^{q,\infty}(Q_{t_{1}})}^{\sigma}
\label{after lemma 2}\\
& \quad  +  \left\Vert c\right\Vert _{L^{r,1}(Q_{t_{1}})}\left\Vert
1+\left\vert u\right\vert +\left\vert v\right\vert \right\Vert _{L^{\overline
{q},\infty}(Q_{t_{1}})}^{\tau} \big] \nonumber \\
& \quad \times
\left\Vert \left\vert \nabla u-\nabla
v\right\vert \right\Vert _{L^{\theta,\infty}(Q_{t_{1}})}  .\nonumber
\end{align}

Under hypotheses (\ref{ip 2}) we can choose $t_{1}$ small enough
such that\ (\ref{prima t1}) holds. Then by (\ref{after lemma 2}) and
(\ref{prima t1}) it follows that for $\theta=\frac{\alpha
(N+2)}{2(N+1)}$
\begin{equation*}
  \left\Vert \left\vert \nabla u-\nabla v\right\vert \right\Vert _{L^{\theta
,\infty}(Q_{t_{1}})}\leq0.
\end{equation*}
Arguing as in Theorem \ref{th2} we conclude that
$u=v$ almost everywhere in $Q_{T}$.
\end{proof}


\medskip
\section*{Acknowledgement}
This work was done during the visits made by the first two authors to
Laboratoire de Math\'ematiques ``Rapha\"el Salem'' de l'Universit\'e
de Rouen and by the third author to Dipartimento di Matematica  della Seconda Universit\`a degli Studi di
Napoli. Hospitality and support of all these
institutions are gratefully acknowledged.

\par\medskip
\bibliographystyle{plain}
\bibliography{unicite}

\end{document}